\documentclass[11pt,a4paper]{article}

\usepackage[margin=1in]{geometry}
\usepackage[T1]{fontenc}
\usepackage[utf8]{inputenc}
\usepackage{lmodern}
\usepackage{amsmath,amssymb,amsfonts,amsthm,mathtools,mathrsfs,bm}
\usepackage{enumitem}
\usepackage{tikz}
\usetikzlibrary{arrows.meta,calc,positioning}
\usepackage{xcolor}
\usepackage{aliascnt}
\usepackage[
colorlinks=true,
linkcolor=blue,
citecolor=red,
urlcolor=blue
]{hyperref}

\usepackage[nameinlink,capitalize]{cleveref}

\newcommand{\G}{\mathcal G}
\newcommand{\K}{\mathcal K}
\newcommand{\E}{\mathcal E}

\newcommand{\Pp}{\mathcal P}
\newcommand{\D}{\mathcal D}
\newcommand{\N}{\mathcal N}
\newcommand{\Y}{Y}
\newcommand{\R}{\mathbb R}
\newcommand{\C}{\mathbb C}
\newcommand{\dd}{\,\mathrm d}
\newcommand{\eps}{\varepsilon}
\newcommand{\ess}{\mathrm{ess}}
\newcommand{\supp}{\operatorname{supp}}
\newcommand{\dist}{\operatorname{dist}}
\newcommand{\dom}{\operatorname{dom}}
\newcommand{\Ree}{\operatorname{Re}}

\newcommand{\abs}[1]{\lvert #1\rvert}
\newcommand{\norm}[1]{\lVert #1\rVert}

\tikzset{
	vertex/.style={circle,fill=black,inner sep=1.7pt},
	openvertex/.style={circle,draw=black,fill=white,inner sep=1.7pt},
	edge/.style={line width=0.8pt},
	halfray/.style={line width=0.8pt,-{Latex[length=2mm]}},
	pendant/.style={line width=1.15pt,blue!70!black},
	core/.style={line width=1pt,orange!80!black},
	label/.style={font=\small}
}

\numberwithin{equation}{section}

\theoremstyle{plain}
\newtheorem{theorem}{Theorem}[section]
\newaliascnt{proposition}{theorem}
\newtheorem{proposition}[proposition]{Proposition}
\aliascntresetthe{proposition}
\newaliascnt{lemma}{theorem}
\newtheorem{lemma}[lemma]{Lemma}
\aliascntresetthe{lemma}
\newaliascnt{corollary}{theorem}
\newtheorem{corollary}[corollary]{Corollary}
\aliascntresetthe{corollary}

\theoremstyle{definition}
\newaliascnt{definition}{theorem}
\newtheorem{definition}[definition]{Definition}
\aliascntresetthe{definition}
\newaliascnt{assumption}{theorem}

\aliascntresetthe{assumption}
\newaliascnt{example}{theorem}
\newtheorem{example}[example]{Example}
\aliascntresetthe{example}

\theoremstyle{remark}
\newaliascnt{remark}{theorem}
\newtheorem{remark}[remark]{Remark}
\aliascntresetthe{remark}

\crefname{theorem}{Theorem}{Theorems}
\Crefname{theorem}{Theorem}{Theorems}
\crefname{proposition}{Proposition}{Propositions}
\Crefname{proposition}{Proposition}{Propositions}
\crefname{lemma}{Lemma}{Lemmas}
\Crefname{lemma}{Lemma}{Lemmas}
\crefname{corollary}{Corollary}{Corollaries}
\Crefname{corollary}{Corollary}{Corollaries}
\crefname{definition}{Definition}{Definitions}
\Crefname{definition}{Definition}{Definitions}
\crefname{assumption}{Assumption}{Assumptions}
\Crefname{assumption}{Assumption}{Assumptions}
\crefname{example}{Example}{Examples}
\Crefname{example}{Example}{Examples}
\crefname{remark}{Remark}{Remarks}
\Crefname{remark}{Remark}{Remarks}

\title{Strict ground-level gaps on long-pendant generalized star graphs for nonlinear Dirac equations}

\author{Zhipeng Yang\thanks{Corresponding author: yangzhipeng326@163.com. }\\
	{\small Yunnan Key Laboratory of Modern Analytical Mathematics and Applications, Kunming 650500, China.}\\
	{\small Department of Mathematics, Yunnan Normal University, Kunming 650500, China.}   
}
\date{}

\begin{document}
	\maketitle
	
	\begin{abstract}
		We study strict ground-level gaps for autonomous nonlinear Dirac equations on noncompact metric graphs.  The relevant level at infinity is the autonomous full-line level \(d_\lambda^\infty=d_{\lambda,\R}\), since a concentrating sequence escaping along a half-line sees the real line after recentering far from all vertices.  We prove that, uniformly for \(\lambda\) in a compact subinterval \(I\Subset(-mc^2,mc^2)\), the presence of a sufficiently long pendant path in the essential reduction of the graph implies
		\[
		d_\lambda(\G)<d_\lambda^\infty .
		\]
		The proof uses a half-line restriction of a symmetric full-line ground state, uniform exponential decay, resolvent localization for massive Dirac-Kirchhoff spectral projections, and a compactness argument for generalized Nehari fibres.  We also identify elementary graph classes for which no such gap can hold: pure transmission graphs whose essential reduction is isometric to the full line satisfy \(d_\lambda(\G)=d_\lambda^\infty\).
	\end{abstract}
	
	\textbf{Keywords.} Nonlinear Dirac equation; generalized star graph; long pendant.
	
	\textbf{2020 Mathematics Subject Classification.} 35Q41, 35R02, 81Q35, 47J30, 34B45.

	
	\section{Introduction and main results}
	
	Metric graphs provide a framework for one-dimensional quantum systems with a prescribed network structure. Differential equations are imposed on the edges, while vertex conditions couple the edge dynamics. General accounts of quantum graph operators and evolution equations on networks can be found in \cite{BerkolaikoKuchment2013,ExnerKeatingKuchmentSunadaTeplyaev2008,Mugnolo2014}. For nonlinear Schr\"odinger equations, the existence of ground states on noncompact graphs depends on both topology and metric lengths \cite{Adami2016,AdamiSerraTilli2015,AdamiSerraTilli2016}. A comparison with the ground level on the real line is central to this dependence, since minimizing sequences may escape along an external edge.
	
	For the focusing power NLS equation, a standard problem is to minimize the energy at prescribed mass,
	\[
	E_q(v)=\frac12\int_\G |v'|^2\,\dd x
	-\frac1q\int_\G |v|^q\,\dd x,
	\qquad v\in H^1(\G),\quad \norm{v}_{L^2(\G)}^2=\mu,
	\qquad 2<q<6.
	\]
	Beyond global minimization, localization constraints yield multiple positive bound states associated with prescribed bounded edges when the mass is sufficiently large \cite{AST2019}. Uniqueness and compactness at prescribed mass have also been studied for graphs consisting of half-lines and a pendant \cite{DST2020}. In a different regime, terminal edges support positive low-action solutions concentrating near terminal vertices at large frequency on compact graphs \cite{DGMP2020}. For cubic NLS, the construction and comparison of states concentrated on individual bounded edges reveal the role of edge type and length \cite{BMP2021}. These results explain why a terminal branch is a natural place to seek a localized profile.
	
	These comparisons depend on the choice of variational level. Prescribed-mass energy minimization and fixed-frequency action minimization on the Nehari manifold are related but distinct problems \cite{DST2023}. On metric graphs, minimization on the Nehari manifold and minimization among nonzero solutions can even have different levels and attainment properties \cite{DDGS2023}. Recent classifications on the \(\mathcal T\)-graph, formed by two half-lines and one pendant \cite{ACT2024}, and on single-knot graphs with half-lines, loops and pendants \cite{ACT2025}, further describe how positive solutions and least-action states depend on the metric parameters. These are scalar counterparts of the geometric question considered here. For the Dirac equation, the positive and negative spectral subspaces both have infinite dimension, and the action is strongly indefinite.
	
	The linear theory of Dirac operators on graphs includes the study of self-adjoint vertex conditions by Bulla and Trenkler \cite{BullaTrenkler1990} and spectral analysis by Bolte and Harrison \cite{BolteHarrison2003}. For the massive two-component realization used here, Borrelli, Carlone and Tentarelli \cite{BCTKirchoff2021} discuss the spectral properties and quadratic-form domain. Complementary boundary and index theory for first-order graph operators, including scalar Dirac realizations, was developed by Richtsfeld \cite{Richtsfeld2024}.
	
	Nonlinear Dirac equations on networks have been studied with several interactions and vertex conditions. Sabirov et al.\ \cite{SabirovEtAl2018} constructed soliton profiles and studied transmission for a Gross--Neveu interaction with weighted junction conditions. For the Dirac--Kirchhoff operator and power nonlinearities, Borrelli, Carlone and Tentarelli \cite{BorrelliCarloneTentarelli2019} obtained infinitely many bound states when the nonlinearity is confined to the compact core, together with a nonrelativistic limit for the constructed states in the range \(2<p<6\). Their subsequent work \cite{BorrelliCarloneTentarelli2021} established local dynamics on noncompact graphs and standing-wave bifurcation on infinite stars. More recently, He and Ji \cite{HeJi2026} studied prescribed-charge solutions with localized power nonlinearities under coupling or topological assumptions, while Xing and Yang \cite{XingYang2026} established local well-posedness on stars for bounded initial data of low Sobolev regularity and powers \(p\ge3\). They also proved local well-posedness at low Sobolev regularity for a quadratic interaction between the positive and negative spectral components on \(N\)-star graphs \cite{XingYangQuadratic2026}. The question addressed here concerns a fixed frequency and a nonlinearity acting throughout the graph: how does a long terminal branch affect the generalized Nehari ground level relative to the full-line level?
	
	Let \(\G\) be a connected noncompact metric graph with finitely many half-lines and compact core.  On each oriented edge we consider the one-dimensional Dirac expression
	\[
	\D=-ic\sigma_1\frac{d}{dx}+mc^2\sigma_3,
	\qquad m>0,\quad c>0,
	\]
	endowed with Dirac-Kirchhoff vertex conditions.  For
	\[
	\lambda\in(-mc^2,mc^2),
	\]
	we study the autonomous nonlinear Dirac equation
	\begin{equation*}
		\D_\G u+\lambda u=f(\abs{u})u .
	\end{equation*}
	The positive and negative spectral subspaces of the massive Dirac operator are both infinite dimensional, so the variational problem is strongly indefinite.  We use the positive-negative spectral splitting of \(\D_\G\) and minimize the associated action functional on a generalized Nehari manifold.  The corresponding ground level is denoted by \(d_\lambda(\G)\).
	
	Throughout the paper we assume that
	\[
	F(t)=\int_0^t f(s)s\,\dd s
	\qquad(t\ge0),
	\]
	and that the nonlinearity \(f\in C^1([0,\infty))\) satisfies the following hypotheses:
	\begin{enumerate}[label=\textnormal{(F\arabic*)},leftmargin=2.2em]
		\item \(f(0)=0\), and \(f(t)\ge0\) for \(t\ge0\);
		\item there exist \(p>2\) and \(C>0\) such that
		\[
		f(t)\le C(1+t^{p-2})
		\qquad(t\ge0);
		\]
		\item there exists \(\theta>2\) such that
		\[
		0<\theta F(t)\le f(t)t^2
		\qquad(t>0);
		\]
		\item the map \(t\mapsto f(t)\) is nondecreasing on \((0,\infty)\).
	\end{enumerate}
	
	The compactness obstruction for minimizing sequences on a noncompact graph is escape along an external edge.  If a profile travels far out on a half-line, then after recentering it sees the full line, not a half-line with an endpoint condition.  Thus the relevant limiting level at infinity is the full-line autonomous ground level
	\[
	d_\lambda^\infty=d_{\lambda,\R}.
	\]
	The strict comparison
	\begin{equation}\label{eq1.1}
		d_\lambda(\G)<d_\lambda^\infty
	\end{equation}
	is the energy separation which prevents this escape mechanism.  In applications to nonautonomous problems, such a strict autonomous comparison is precisely the kind of gap used to recover compactness of Palais-Smale or minimizing sequences.
	
	The purpose of this paper is to give a concrete geometric condition implying \eqref{eq1.1}.  The condition is the existence of a sufficiently long pendant path in the essential reduction of the graph.  The essential reduction is obtained by suppressing purely transmissive degree-two vertices, and a pendant path is a compact terminal path attached to the rest of the graph through a single endpoint.  The precise definitions are given in \cref{Sec2}.
	
	The terminal endpoint of a pendant path carries the degree-one Dirac-Kirchhoff condition \(u_2=0\).  This is compatible with the odd lower component of a suitably centered full-line ground state.  Hence one may place one half of a symmetric full-line ground state on a long pendant, truncate it near the attachment point, and use it as a below-threshold test profile.  The nonlocal point is that the test profile must be compared after spectral projection onto the positive subspace of the graph Dirac operator.  Since spectral projections are nonlocal, it is not enough to compare the differential expression locally on the pendant.  A main part of the proof is therefore devoted to showing that, as the pendant length tends to infinity, the localized graph spectral projections and the corresponding generalized Nehari fibre maxima converge to those of the endpoint half-line model.
	
	Variational approaches to nonlinear Dirac equations in Euclidean space were developed by Esteban and S\'er\'e \cite{EstebanSere1995} and Bartsch and Ding \cite{BartschDing2006}; broader accounts are given in \cite{EstebanSere2002,EstebanLewinSere2008}. The positive--negative spectral splitting and generalized Nehari reduction belong to this strongly indefinite framework \cite{SzulkinWeth2010}. Concentration of semiclassical states for scalar radial interactions is treated by Ding and Xu in three dimensions \cite{DingXu2015} and by Ding, Guo and Xu in dimensions \(n\ge2\) \cite{DingGuoXu2021}. For the present one-dimensional two-component problem, \cref{AppA} verifies the autonomous variational construction directly on \(H^{1/2}(\R,\C^2)\) under (F1)--(F4).
	
	In the Euclidean autonomous problem, translation of profiles causes the loss of compactness. On a noncompact graph, a profile escaping along an external edge has the full-line problem as its recentered limit. The additional step in the present comparison is to localize the nonlocal spectral projections and the generalized Nehari fibres near a terminal branch. This connects the endpoint half-line profile to the graph functional, with estimates independent of the remaining compact core, and yields a uniform strict comparison with the full-line level.
	
	Our main result is the following uniform strict-gap theorem.
	
	\begin{theorem}\label{Thm1.1}
		Let \(I\) be a compact interval with \(I\Subset(-mc^2,mc^2)\), and assume that \(f\) satisfies \textnormal{(F1)--(F4)}.  Then there exists \(L_I>0\) such that, if the essential reduction of a generalized star graph \(\G\) contains a pendant path of length \(L\ge L_I\), then
		\begin{equation*}
			d_\lambda(\G)<d_\lambda^\infty
			\qquad\text{for every }\lambda\in I .
		\end{equation*}
		Moreover, the inequality is uniform on \(I\): there exists \(\rho_I>0\) such that
		\begin{equation}\label{eq1.2}
			d_\lambda(\G)\le d_\lambda^\infty-\rho_I
			\qquad\text{for every }\lambda\in I .
		\end{equation}
	\end{theorem}
	
	The result says that a sufficiently long terminal branch produces a genuine sub-full-line test level.  The constants \(L_I\) and \(\rho_I\) depend only on \(I,m,c,f\) and are independent of the shortest bounded-edge length, the vertex degrees, the number of edges, and the topology of the compact core away from the distinguished pendant path.
	
	The next result shows that a compact presentation of a graph does not by itself create a strict gap.  If the essential reduction is the real line, then the graph is analytically equivalent to the full-line model.
	
	\begin{theorem}\label{Thm1.2}
		Let \(\G\) be a generalized star graph whose essential reduction is isometric to \(\R\).  Then
		\[
		d_\lambda(\G)=d_\lambda^\infty
		\qquad\text{for every }\lambda\in(-mc^2,mc^2).
		\]
	\end{theorem}
	
	We also gived the endpoint analogue.  It clarifies why finite tails attached to a single half-line should not be confused with the full-line escape threshold.
	
	For the endpoint model, let
	\[
	\D_+u=-ic\sigma_1u'+mc^2\sigma_3u,
	\qquad
	\dom(\D_+)
	=
	\left\{
	u=(u_1,u_2)^T\in H^1(\R_+,\C^2):u_2(0)=0
	\right\}.
	\]
	Let
	\[
	\Y_{\R_+}
	=
	\dom(|\D_+|^{1/2})
	=
	\Y_{\R_+}^+\oplus\Y_{\R_+}^-,
	\qquad
	\norm{u}_+^2
	:=
	\norm{|\D_+|^{1/2}u}_{L^2(\R_+)}^2 .
	\]
	For \(u=u^++u^-\in\Y_{\R_+}\), define
	\[
	\begin{split}
		J_{\lambda,\R_+}(u)=
		\frac12\left(\norm{u^+}_+^2-\norm{u^-}_+^2\right)
		+\frac{\lambda}{2}\norm{u}_{L^2(\R_+)}^2-
		\int_0^\infty F(\abs{u})\,\dd x,
	\end{split}
	\]
	and
	\[
	\N_{\lambda,\R_+}
	=
	\left\{
	u\in\Y_{\R_+}\setminus\Y_{\R_+}^-:
	J'_{\lambda,\R_+}(u)[u]=0,\quad
	J'_{\lambda,\R_+}(u)[\eta]=0
	\ \text{for every }\eta\in\Y_{\R_+}^-
	\right\}.
	\]
	The half-line ground level is
	\[
	d_\lambda^+
	:=
	\inf_{u\in\N_{\lambda,\R_+}}
	J_{\lambda,\R_+}(u).
	\]
	
	\begin{proposition}\label{Prop1.3}
		Let \(\G\) be a generalized star graph whose essential reduction is isometric to \(\R_+\).  Then
		\[
		d_\lambda(\G)=d_\lambda^+
		\qquad\text{for every }\lambda\in(-mc^2,mc^2).
		\]
	\end{proposition}
	
	\begin{remark}
		The level \(d_\lambda^+\) is not the escape level of a Palais-Smale sequence going to infinity along an external edge.  It is the endpoint model level.  The escape level is \(d_\lambda^\infty=d_{\lambda,\R}\).
	\end{remark}
	
	The paper is organized as follows.  In \cref{Sec2} we introduce generalized star graphs, essential reductions and pendant paths, list representative examples, and develop the variational framework, including the Dirac-Kirchhoff operator, the spectral splitting and the generalized Nehari reduction.  The long-pendant construction and the localized fibre convergence are proved in \cref{Sec3}.  The proofs of the main results are completed in the final section.  The one-dimensional full-line ground-state result is justified in the appendix.

	\section{Graphs, examples and variational setting}\label{Sec2}
	
	We collect in this section the graph-theoretic notation and the variational framework used throughout the paper.  Metric graphs and quantum graph operators are standard objects; see \cite{BerkolaikoKuchment2013,ExnerKeatingKuchmentSunadaTeplyaev2008,Mugnolo2014} for general background.  In the present paper the relevant graphs are noncompact metric graphs with finitely many external edges and compact core.  The vertex conditions are Dirac-Kirchhoff conditions.  Closely related Dirac-Kirchhoff realizations on metric graphs are considered in \cite{BorrelliCarloneTentarelli2019,BorrelliCarloneTentarelli2021}.
	
	\subsection{Generalized star graphs and essential reductions}
	
	\begin{definition}\label{Def2.1}
		A connected noncompact metric graph \(\G\) is called a generalized star graph if it has finitely many bounded edges, its compact core \(\K\), defined as their union, is nonempty, and there are finitely many half-lines \(e_1,\ldots,e_M\) with pairwise disjoint interiors such that
		\[
		\G=\K\cup\bigcup_{j=1}^M e_j,
		\qquad
		\K\cap e_j=\{v_j\}.
		\]
	\end{definition}
	
	Thus a generalized star graph is a finite compact graph with finitely many half-lines attached.  The compact core may contain loops, multiple edges, finite trees, and terminal bounded edges.  Similar graph classes appear in the variational theory of nonlinear Schr\"odinger equations on noncompact metric graphs; see \cite{Adami2016,AdamiSerraTilli2015,AdamiSerraTilli2016}.
	
	A degree-two vertex at which the Dirac-Kirchhoff condition is purely transmissive is artificial from the metric point of view: it only subdivides an edge.  Suppressing such vertices and concatenating the adjacent edges gives the essential reduction.  Whenever a chain is concatenated, orientations are chosen consistently along the chain.  If an original edge orientation is reversed during this identification, we use the unitary change of variables
	\begin{equation}\label{eq2.1}
		(u^1(x),u^2(x))\longmapsto (u^1(\ell-x),-u^2(\ell-x)).
	\end{equation}
	This transformation preserves \(\abs{u}\) and intertwines the Dirac expression and the Dirac-Kirchhoff vertex condition with their forms in the reoriented coordinate.
	
	\begin{definition}\label{Def2.2}
		The essential reduction \(\G_{\ess}\) of a generalized star graph \(\G\) is the metric graph obtained by suppressing every transmissive vertex of degree two and replacing each resulting chain of edges by one edge of the same total length, with the orientation convention described above.
	\end{definition}
	
	The essential reduction changes only the combinatorial presentation of a purely transmissive chain.  It preserves the total metric length and the Dirac-Kirchhoff dynamics along the concatenated coordinate.
	
	\begin{figure}[t]
		\centering
		\begin{tikzpicture}[scale=0.95,
			mainlabel/.style={font=\small,align=center},
			smalllabel/.style={font=\scriptsize,align=center}
			]
			\draw[edge] (-5.2,0)--(-4.0,0);
			\draw[edge] (-4.0,0)--(-2.7,0);
			\draw[edge] (-2.7,0)--(-1.4,0);
			\node[openvertex] at (-5.2,0) {};
			\node[vertex] at (-4.0,0) {};
			\node[vertex] at (-2.7,0) {};
			\node[vertex] at (-1.4,0) {};
			\node[mainlabel] at (-3.3,0.48) {transmissive\\degree-two chain};
			\node[smalllabel] at (-3.3,-0.52) {lengths \(\ell_1,\ell_2,\ell_3\)};
			
			\draw[-{Latex[length=2mm]},line width=0.7pt] (-0.75,0)--(0.55,0);
			\node[smalllabel] at (-0.10,0.38) {suppress};
			
			\draw[edge] (1.15,0)--(5.2,0);
			\node[openvertex] at (1.15,0) {};
			\node[vertex] at (5.2,0) {};
			\node[mainlabel] at (3.2,0.48) {one edge of length\\\(\ell_1+\ell_2+\ell_3\)};
			\node[smalllabel] at (3.2,-0.52) {\(\G_{\ess}\)};
		\end{tikzpicture}
		\caption{Suppression of purely transmissive degree-two vertices.  The metric length is preserved, and the orientation convention \eqref{eq2.1} keeps the Dirac expression in the standard form.}
		\label{Fig1}
	\end{figure}
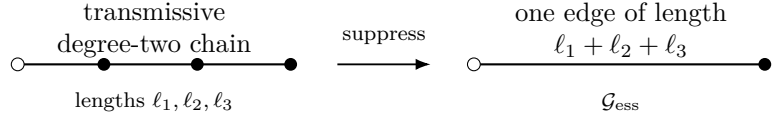
	
	\begin{definition}\label{Def2.3}
		A compact path \(\Pp\subset\G_{\ess}\) is called a pendant path if \(\Pp\) is isometric to \([0,L]\), one endpoint of \(\Pp\) is a terminal vertex, and
		\[
		\Pp\cap\overline{\G_{\ess}\setminus\Pp}
		\]
		consists exactly of the other endpoint.  The number \(L\) is called the pendant length.
	\end{definition}
	
	The second endpoint is not required to be a vertex; it may lie in the interior of a bounded or external edge.
	
	Equivalently, a pendant path is a terminal compact arm whose interior does not meet any other edge of the essential reduction.  In a finite tree attached to the rest of the graph, a leaf-to-branch arm may be a pendant path, but a longer route with a side branch attached at an interior point is not a pendant path.
	
	\begin{definition}\label{Def2.4}
		Let \(L_*>0\).  A generalized star graph \(\G\) is called \(L_*\)-long-pendant if its essential reduction contains a pendant path of length \(L\ge L_*\). Once a compact interval \(I\Subset(-mc^2,mc^2)\) has been fixed, the term \(I\)-long-pendant means \(L_I\)-long-pendant, where \(L_I\) is the threshold in \cref{Thm1.1}.
	\end{definition}
	
	By undoing the suppression of transmissive degree-two vertices, a pendant path in \(\G_{\ess}\) is represented in \(\G\) by a terminal metric interval of the same length.  This interval is a finite concatenation of edge segments separated only by transmissive degree-two vertices; its second endpoint may be a vertex or an interior point of a bounded or external edge.  We identify it with \([0,L]\), using \eqref{eq2.1} whenever an edge has to be reoriented.  A spinor which is an \(H^1\)-function of the concatenated coordinate satisfies the matching condition at every suppressed vertex: the first component is continuous and the signed lower-component flux is continuous across the vertex.
	
	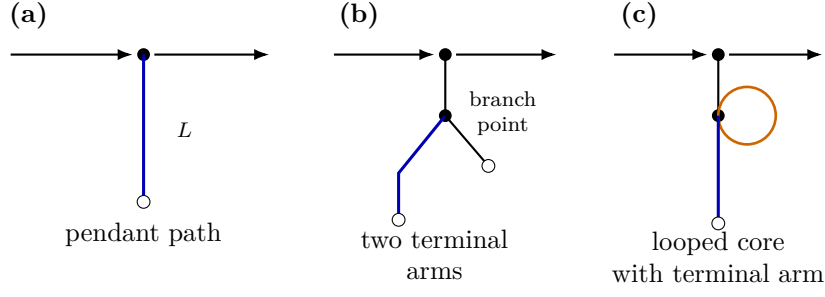
\begin{figure}[t]
		\centering
		\begin{tikzpicture}[scale=0.95,
			panel/.style={font=\small\bfseries},
			mainlabel/.style={font=\small,align=center},
			smalllabel/.style={font=\scriptsize,align=center}
			]
			\node[panel] at (-5.15,0.55) {(a)};
			\draw[halfray] (-5.4,0)--(-3.7,0);
			\draw[halfray] (-3.4,0)--(-1.8,0);
			\node[vertex] at (-3.55,0) {};
			\draw[pendant] (-3.55,0)--(-3.55,-2.05);
			\node[openvertex] at (-3.55,-2.05) {};
			\node[mainlabel] at (-3.55,-2.50) {pendant path};
			\node[smalllabel] at (-2.98,-1.05) {\(L\)};
			
			\node[panel] at (-0.55,0.55) {(b)};
			\draw[halfray] (-0.9,0)--(0.5,0);
			\draw[halfray] (0.8,0)--(2.2,0);
			\node[vertex] at (0.65,0) {};
			\draw[edge] (0.65,0)--(0.65,-0.85);
			\node[vertex] at (0.65,-0.85) {};
			\draw[pendant] (0.65,-0.85)--(0.0,-1.65)--(0.0,-2.30);
			\node[openvertex] at (0.0,-2.30) {};
			\draw[edge] (0.65,-0.85)--(1.25,-1.55);
			\node[openvertex] at (1.25,-1.55) {};
			\node[smalllabel] at (1.45,-0.82) {branch\\point};
			\node[mainlabel] at (0.50,-2.78) {two terminal\\arms};
			
			\node[panel] at (3.35,0.55) {(c)};
			\draw[halfray] (3.0,0)--(4.3,0);
			\draw[halfray] (4.6,0)--(5.9,0);
			\node[vertex] at (4.45,0) {};
			\draw[edge] (4.45,0)--(4.45,-0.85);
			\node[vertex] at (4.45,-0.85) {};
			\draw[core] (4.85,-0.85) circle (0.4);
			\draw[pendant] (4.45,-0.85)--(4.45,-2.35);
			\node[openvertex] at (4.45,-2.35) {};
			\node[mainlabel] at (4.45,-2.82) {looped core\\with terminal arm};
		\end{tikzpicture}
		\caption{Pendant paths in the essential reduction.  Panel (a) shows a single pendant path, panel (b) shows two terminal arms meeting at a branch point, and panel (c) shows that decorations at the attachment endpoint are allowed.}
		\label{Fig2}
	\end{figure}
	
	\subsection{Examples}
	
	We now list graph configurations to which \cref{Thm1.1} applies once the displayed terminal path has length at least \(L_I\).  The relevant pendant paths are drawn in blue.
	
	\begin{example}\label{Ex2.5}
		\emph{\(T\)-graph.} Let \(\G\) be obtained by attaching two half-lines and one bounded edge of length \(L\) at a common vertex.  The bounded edge ends at a terminal vertex. This is the graph shown in panel (a) of \cref{Fig3}.  If \(L\ge L_I\), then \(\G\) is \(I\)-long-pendant and
		\[
		d_\lambda(\G)<d_\lambda^\infty
		\qquad(\lambda\in I).
		\]
	\end{example}
	
	\begin{example}\label{Ex2.6}
		\emph{Line with one long terminal branch.} Let a compact path of length \(L\) be attached to one point of the real line, and assume that the free endpoint of this path is terminal.  This is shown in panel (b) of \cref{Fig3}.  If \(L\ge L_I\), then the strict gap holds on \(I\).
	\end{example}
	
	\begin{example}\label{Ex2.7}
		\emph{Terminal arm of a finite tree.} Suppose that \(\G_{\ess}\) contains a finite tree attached to the rest of the graph through a single root vertex.  If the tree contains a terminal arm, from the root or from a branching vertex to a leaf, whose interior has no side branch and whose length is at least \(L_I\), then \(\G\) is \(I\)-long-pendant.  A representative picture is panel (c) of \cref{Fig3}.
	\end{example}
	
	\begin{example}\label{Ex2.8}
		\emph{Looped or decorated compact core with a long terminal arm.} Let a compact decorated branch, possibly containing loops or finite side edges, be attached to a noncompact trunk.  If one of the terminal arms issuing from this compact branch has length at least \(L_I\), then the strict gap follows from \cref{Thm1.1}.  This configuration is shown in panel (d) of \cref{Fig3}.
	\end{example}
	
	\begin{figure}[t]
		\centering
		\begin{tikzpicture}[scale=0.88,
			panel/.style={font=\small\bfseries},
			mainlabel/.style={font=\small,align=center},
			smalllabel/.style={font=\scriptsize,align=center}
			]
			\node[panel] at (-6.0,0.55) {(a)};
			\draw[halfray] (-6.2,0)--(-4.65,0);
			\draw[halfray] (-4.35,0)--(-2.85,0);
			\node[vertex] at (-4.5,0) {};
			\draw[pendant] (-4.5,0)--(-4.5,-1.9);
			\node[openvertex] at (-4.5,-1.9) {};
			\node[mainlabel] at (-4.5,-2.35) {\(T\)-graph};
			
			\node[panel] at (-1.95,0.55) {(b)};
			\draw[halfray] (-2.15,0)--(-0.75,0);
			\draw[halfray] (-0.45,0)--(0.95,0);
			\node[vertex] at (-0.6,0) {};
			\draw[pendant] (-0.6,0)--(-0.6,-1.9);
			\node[openvertex] at (-0.6,-1.9) {};
			\node[mainlabel] at (-0.6,-2.35) {line with\\terminal branch};
			
			\node[panel] at (1.85,0.55) {(c)};
			\draw[halfray] (1.65,0)--(3.0,0);
			\draw[halfray] (3.3,0)--(4.65,0);
			\node[vertex] at (3.15,0) {};
			\draw[edge] (3.15,0)--(3.15,-0.65);
			\node[vertex] at (3.15,-0.65) {};
			\draw[pendant] (3.15,-0.65)--(2.45,-1.35)--(2.45,-2.05);
			\node[openvertex] at (2.45,-2.05) {};
			\draw[edge] (3.15,-0.65)--(3.85,-1.35);
			\node[openvertex] at (3.85,-1.35) {};
			\node[mainlabel] at (3.15,-2.65) {terminal arm\\in a tree};
			
			\node[panel] at (5.55,0.55) {(d)};
			\draw[halfray] (5.35,0)--(6.65,0);
			\draw[halfray] (6.95,0)--(8.25,0);
			\node[vertex] at (6.8,0) {};
			\draw[edge] (6.8,0)--(6.8,-0.75);
			\node[vertex] at (6.8,-0.75) {};
			\draw[core] (7.17,-0.75) circle (0.37);
			\draw[pendant] (6.8,-0.75)--(6.8,-2.05);
			\node[openvertex] at (6.8,-2.05) {};
			\node[mainlabel] at (6.8,-2.65) {decorated\\compact core};
		\end{tikzpicture}
		\caption{Basic positive configurations.  Panels (a)--(d) correspond respectively to \cref{Ex2.5,Ex2.6,Ex2.7,Ex2.8}.}
		\label{Fig3}
	\end{figure}
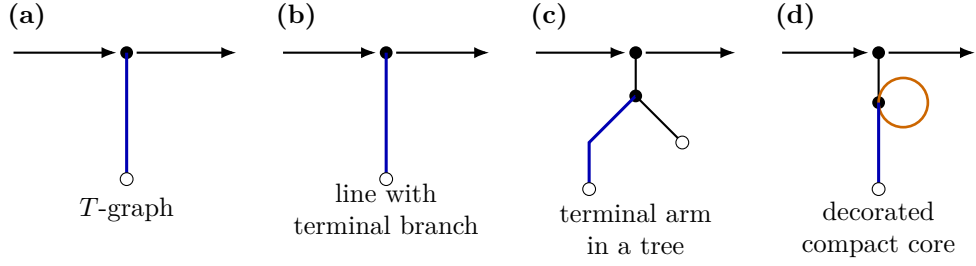
	
	\begin{example}\label{Ex2.9}
		\emph{Several terminal arms.} If \(\G_{\ess}\) contains several pendant paths and at least one of them has length at least \(L_I\), then \(\G\) is \(I\)-long-pendant.  The proof uses only one such terminal arm.  This situation is shown in panel (a) of \cref{Fig4}.
	\end{example}
	
	\begin{example}\label{Ex2.10}
		\emph{Pendant chain with artificial vertices.} Suppose a terminal branch is represented in the original graph by a chain of bounded edges separated by transmissive degree-two vertices, with total length \(L\).  Its essential reduction is a single pendant path of length \(L\).  This reduction is shown in panels (b) and (c) of \cref{Fig4}.  Hence the strict gap holds on \(I\) if \(L\ge L_I\).
	\end{example}
	
	\begin{figure}[t]
		\centering
		\begin{tikzpicture}[scale=0.95,
			panel/.style={font=\small\bfseries},
			mainlabel/.style={font=\small,align=center},
			smalllabel/.style={font=\scriptsize,align=center}
			]
			\node[panel] at (-5.35,0.55) {(a)};
			\draw[halfray] (-5.6,0)--(-4.45,0);
			\draw[halfray] (-4.15,0)--(-3.0,0);
			\node[vertex] at (-4.3,0) {};
			\draw[pendant] (-4.3,0)--(-4.95,-0.7)--(-4.95,-1.55);
			\node[openvertex] at (-4.95,-1.55) {};
			\draw[pendant] (-4.3,0)--(-3.65,-0.7)--(-3.65,-1.90);
			\node[openvertex] at (-3.65,-1.90) {};
			\node[mainlabel] at (-4.30,-2.35) {several\\terminal arms};
			
			\node[panel] at (-1.55,0.55) {(b)};
			\draw[halfray] (-1.8,0)--(-0.55,0);
			\draw[halfray] (-0.25,0)--(1.00,0);
			\node[vertex] at (-0.40,0) {};
			\draw[pendant] (-0.40,0)--(-0.40,-0.55);
			\node[vertex] at (-0.40,-0.55) {};
			\draw[pendant] (-0.40,-0.55)--(-0.40,-1.10);
			\node[vertex] at (-0.40,-1.10) {};
			\draw[pendant] (-0.40,-1.10)--(-0.40,-1.95);
			\node[openvertex] at (-0.40,-1.95) {};
			\node[mainlabel] at (-0.40,-2.38) {chain before\\reduction};
			\node[smalllabel] at (0.30,-1.05) {degree-two\\vertices};
			
			\draw[-{Latex[length=2mm]},line width=0.7pt] (1.50,-0.95)--(2.65,-0.95);
			\node[smalllabel] at (2.08,-0.55) {\(\G_{\ess}\)};
			
			\node[panel] at (3.15,0.55) {(c)};
			\draw[halfray] (3.00,0)--(4.25,0);
			\draw[halfray] (4.55,0)--(5.80,0);
			\node[vertex] at (4.40,0) {};
			\draw[pendant] (4.40,0)--(4.40,-1.95);
			\node[openvertex] at (4.40,-1.95) {};
			\node[mainlabel] at (4.40,-2.38) {single\\pendant path};
		\end{tikzpicture}
		\caption{Additional positive configurations.  Panel (a) corresponds to \cref{Ex2.9}.  Panels (b)--(c) correspond to \cref{Ex2.10}: a chain with transmissive degree-two vertices becomes one pendant path in the essential reduction.}
		\label{Fig4}
	\end{figure}
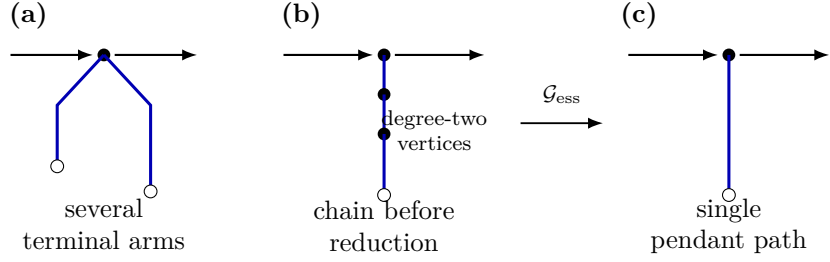
	
	The next two examples illustrate invariance under insertion of transmissive degree-two vertices: full-line reduction gives the level \(d_\lambda^\infty\), whereas half-line reduction gives \(d_\lambda^+\), by \cref{Thm1.2,Prop1.3}. The final two examples illustrate paths excluded by the pendant condition; \cref{Thm1.1} does not apply unless an additional long terminal pendant path is present.
	
	\begin{example}\label{Ex2.11}
		\emph{Finite chain between two half-lines.} Let \(\G\) consist of two half-lines joined by a finite chain of bounded edges, and assume that every interior vertex is transmissive of degree two.  Then \(\G_{\ess}\simeq\R\), as shown in panel (a) of \cref{Fig5}.  Therefore
		\[
		d_\lambda(\G)=d_\lambda^\infty
		\qquad(\lambda\in(-mc^2,mc^2)).
		\]
	\end{example}
	
	\begin{example}\label{Ex2.12}
		\emph{Finite tail followed by one half-line.} Let \(\G\) be a finite path whose right endpoint is attached to a single half-line and whose left endpoint is terminal.  Then \(\G_{\ess}\simeq\R_+\), as shown in panel (b) of \cref{Fig5}, and
		\[
		d_\lambda(\G)=d_\lambda^+,
		\]
		where \(d_\lambda^+\) is the half-line endpoint level.
	\end{example}
	
	\begin{example}\label{Ex2.13}
		\emph{Long internal bridge.} A long compact edge connecting two nonterminal vertices is not a pendant path, even if its length is very large.  This is shown in panel (c) of \cref{Fig5}.  The theorem does not apply to such an internal bridge unless the graph also contains a sufficiently long terminal pendant path.
	\end{example}
	
	\begin{example}\label{Ex2.14}
		\emph{Side branch along a terminal route.} If one tries to take a path from a leaf to a distant root but the path has a side branch attached at an interior point, then that whole route is not a pendant path.  The maximal terminal subpath from the leaf to the first branching vertex is a pendant path.  This distinction is shown in panel (d) of \cref{Fig5}.
	\end{example}
	
	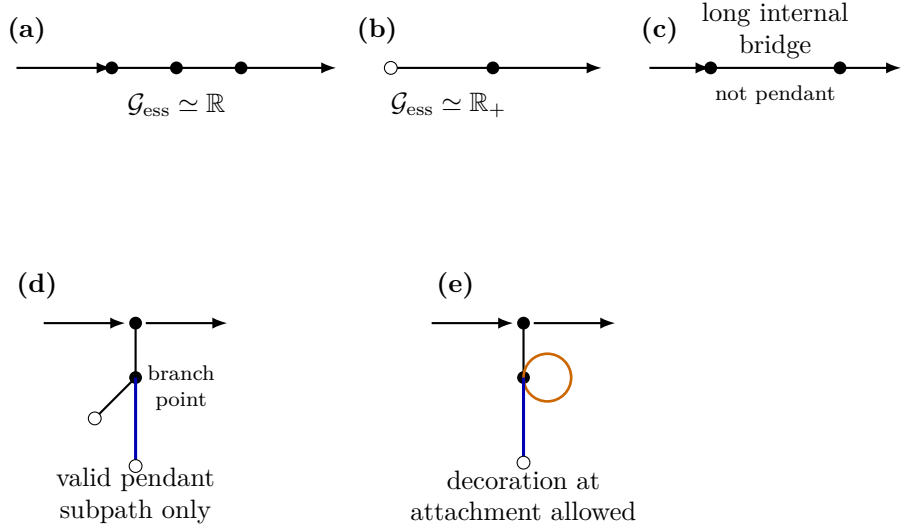
\begin{figure}[t]
		\centering
		\begin{tikzpicture}[scale=0.90,
			panel/.style={font=\small\bfseries},
			mainlabel/.style={font=\small,align=center},
			smalllabel/.style={font=\scriptsize,align=center}
			]
			\node[panel] at (-6.05,0.55) {(a)};
			\draw[halfray] (-6.2,0)--(-4.8,0);
			\draw[edge] (-4.8,0)--(-2.9,0);
			\draw[halfray] (-2.9,0)--(-1.5,0);
			\node[vertex] at (-4.8,0) {};
			\node[vertex] at (-3.85,0) {};
			\node[vertex] at (-2.9,0) {};
			\node[mainlabel] at (-3.85,-0.55) {\(\G_{\ess}\simeq\R\)};
			
			\node[panel] at (-0.90,0.55) {(b)};
			\draw[edge] (-0.7,0)--(0.8,0);
			\draw[halfray] (0.8,0)--(2.4,0);
			\node[openvertex] at (-0.7,0) {};
			\node[vertex] at (0.8,0) {};
			\node[mainlabel] at (0.15,-0.55) {\(\G_{\ess}\simeq\R_+\)};
			
			\node[panel] at (3.25,0.55) {(c)};
			\draw[halfray] (3.1,0)--(4.0,0);
			\draw[halfray] (5.9,0)--(6.8,0);
			\node[vertex] at (4.0,0) {};
			\node[vertex] at (5.9,0) {};
			\draw[edge] (4.0,0)--(5.9,0);
			\node[mainlabel] at (4.95,0.55) {long internal\\bridge};
			\node[smalllabel] at (4.95,-0.40) {not pendant};
			
			\node[panel] at (-5.90,-3.20) {(d)};
			\draw[halfray] (-5.8,-3.75)--(-4.6,-3.75);
			\draw[halfray] (-4.3,-3.75)--(-3.1,-3.75);
			\node[vertex] at (-4.45,-3.75) {};
			\draw[edge] (-4.45,-3.75)--(-4.45,-4.55);
			\node[vertex] at (-4.45,-4.55) {};
			\draw[edge] (-4.45,-4.55)--(-5.05,-5.15);
			\node[openvertex] at (-5.05,-5.15) {};
			\draw[pendant] (-4.45,-4.55)--(-4.45,-5.85);
			\node[openvertex] at (-4.45,-5.85) {};
			\node[smalllabel] at (-3.78,-4.70) {branch\\point};
			\node[mainlabel] at (-4.45,-6.28) {valid pendant\\subpath only};
			
			\node[panel] at (0.25,-3.20) {(e)};
			\draw[halfray] (-0.1,-3.75)--(1.1,-3.75);
			\draw[halfray] (1.4,-3.75)--(2.6,-3.75);
			\node[vertex] at (1.25,-3.75) {};
			\draw[edge] (1.25,-3.75)--(1.25,-4.55);
			\node[vertex] at (1.25,-4.55) {};
			\draw[core] (1.60,-4.55) circle (0.35);
			\draw[pendant] (1.25,-4.55)--(1.25,-5.80);
			\node[openvertex] at (1.25,-5.80) {};
			\node[mainlabel] at (1.25,-6.28) {decoration at\\attachment allowed};
		\end{tikzpicture}
		\caption{Invariant model configurations and paths excluded by the pendant condition.  Panels (a)--(b) correspond to \cref{Ex2.11,Ex2.12}; panel (c) corresponds to \cref{Ex2.13}; panel (d) corresponds to \cref{Ex2.14}. Panel (e) recalls that decorations at the attachment endpoint do not destroy the pendant property.}
		\label{Fig5}
	\end{figure}
	
	\subsection{Consequences of the assumptions on \texorpdfstring{\(f\)}{f}}
	
	Recall that
	\[
	F(t)=\int_0^t f(s)s\,\dd s
	\qquad(t\ge0),
	\]
	and that \(f\) satisfies \textnormal{(F1)--(F4)}.  We gived here some elementary consequences which will be used repeatedly.
	
	Since \(f\) is nondecreasing and nonnegative, \(F\) is nondecreasing and convex on \([0,\infty)\).  Moreover the radial map
	\[
	\xi\longmapsto F(\abs{\xi})
	\]
	is convex on \(\C^2\).  Indeed, \(F'(t)=f(t)t\) and
	\[
	F''(t)=f'(t)t+f(t)\ge0
	\qquad(t>0),
	\]
	while \(F\) is nondecreasing.
	
	Furthermore,
	\begin{equation*}
		H(t):=\frac12 f(t)t^2-F(t)\ge0
		\qquad(t\ge0).
	\end{equation*}
	This follows from
	\[
	F(t)=\int_0^t f(s)s\,\dd s
	\le f(t)\int_0^t s\,\dd s
	=
	\frac12 f(t)t^2 .
	\]
	By \textnormal{(F2)} and \(f(0)=0\), for every \(\eps>0\) there exists \(C_\eps>0\) such that
	\begin{equation}\label{eq2.2}
		f(t)t^2\le \eps t^2+C_\eps t^p,
		\qquad
		F(t)\le \eps t^2+C_\eps t^p
		\qquad(t\ge0).
	\end{equation}
	The Ambrosetti-Rabinowitz condition \textnormal{(F3)} gives
	\begin{equation}\label{eq2.3}
		H(t)
		=
		\frac12 f(t)t^2-F(t)
		\ge
		\left(\frac12-\frac1\theta\right)f(t)t^2>0
		\qquad(t>0).
	\end{equation}
	It also implies a superquadratic lower bound at infinity.  Since
	\[
	\frac{d}{dt}\bigl(F(t)t^{-\theta}\bigr)
	=
	t^{-\theta-1}\bigl(f(t)t^2-\theta F(t)\bigr)\ge0
	\qquad(t>0),
	\]
	the function \(t\mapsto F(t)t^{-\theta}\) is nondecreasing on \((0,\infty)\).  Hence, for every \(r>0\),
	\begin{equation}\label{eq2.4}
		F(t)\ge F(r)r^{-\theta}t^\theta
		\qquad(t\ge r).
	\end{equation}
	
	\subsection{Dirac-Kirchhoff operator and action functional}
	
	Let \(\E\) denote the finite set of edges of \(\G\). On an oriented edge \(e\), write
	\[
	u_e=(u_e^1,u_e^2)^T .
	\]
	The Dirac-Kirchhoff vertex condition is
	\begin{equation}\label{eq2.5}
		u_e^1(v)=u_f^1(v)\quad(e,f\succ v),
		\qquad
		\sum_{e\succ v}\nu_{v,e}u_e^2(v)=0,
	\end{equation}
	Here \(e\succ v\) ranges over the edge ends incident at \(v\), and \(\nu_{v,e}=1\) if the coordinate on that edge end is outgoing from \(v\), while \(\nu_{v,e}=-1\) otherwise. A loop at \(v\) contributes both endpoint traces, with their respective signs. At a terminal vertex this condition reduces to
	\[
	u^2(v)=0.
	\]
	
	The self-adjoint Dirac-Kirchhoff operator is
	\begin{equation*}
		\begin{split}
			\D_\G u&=-ic\sigma_1u'+mc^2\sigma_3u,\\
			\dom(\D_\G)
			&=
			\left\{
			u\in\bigoplus_{e\in\E}H^1(e,\C^2):
			u\text{ satisfies }\eqref{eq2.5}
			\right\}.
		\end{split}
	\end{equation*}
	The boundary form associated with \(-ic\sigma_1d/dx\) vanishes on \eqref{eq2.5}.  The corresponding space of boundary values is maximal isotropic for the boundary symplectic form; equivalently, the standard boundary-triple criterion for first-order systems on finite-degree metric graphs gives self-adjointness.  This is the usual Dirac-Kirchhoff realization considered, for instance, in \cite{BorrelliCarloneTentarelli2019,BorrelliCarloneTentarelli2021}.
	
	For \(u\in\dom(\D_\G)\), integration by parts gives
	\begin{equation*}
		\norm{\D_\G u}_{L^2(\G)}^2
		=
		c^2\norm{u'}_{L^2(\G)}^2
		+
		m^2c^4\norm{u}_{L^2(\G)}^2 .
	\end{equation*}
	Indeed, the boundary terms produced by the first-order part cancel exactly under \eqref{eq2.5}.  Consequently,
	\begin{equation}\label{eq2.6}
		\sigma(\D_\G)\cap(-mc^2,mc^2)=\varnothing.
	\end{equation}
	
	We set
	\[
	\Y_\G=\dom(|\D_\G|^{1/2}),
	\qquad
	\Y_\G=\Y_\G^+\oplus\Y_\G^- ,
	\]
	where \(\Y_\G^\pm\) are the positive and negative spectral subspaces of \(\D_\G\).  We use the norm
	\[
	\norm{u}_\G^2
	:=
	\bigl\langle |\D_\G|^{1/2}u,|\D_\G|^{1/2}u\bigr\rangle_{L^2(\G)} .
	\]
	By \eqref{eq2.6},
	\[
	\norm{u}_{L^2(\G)}
	\le
	(mc^2)^{-1/2}\norm{u}_\G .
	\]
	On every fixed generalized star graph, interpolation gives
	\[
	\Y_\G=[L^2(\G,\C^2),\dom(\D_\G)]_{1/2}
	\hookrightarrow\bigoplus_{e\in\E}H^{1/2}(e,\C^2),
	\]
	where the interpolation norm is equivalent to the spectral form norm; see \cite[Section 4]{BCTKirchoff2021}. Consequently,
	\[
	\Y_\G\hookrightarrow L^q(\G,\C^2)
	\qquad(2\le q<\infty)
	\]
	continuously, and restriction to a compact subgraph is compact into each such \(L^q\) space. The embedding constants here may depend on the fixed graph.
	
	For \(u=u^++u^-\), define
	\begin{equation*}
		J_{\lambda,\G}(u)
		=
		\frac12\bigl(\norm{u^+}_\G^2-\norm{u^-}_\G^2\bigr)
		+
		\frac{\lambda}{2}\int_\G\abs{u}^2\,\dd x
		-
		\int_\G F(\abs{u})\,\dd x .
	\end{equation*}
	Equivalently, the quadratic part is the closed form associated with \(\D_\G+\lambda\).  We denote twice this quadratic part by
	\begin{equation*}
		Q_{\lambda,\G}(u)
		=
		\norm{u^+}_\G^2-\norm{u^-}_\G^2
		+
		\lambda\norm{u}_{L^2(\G)}^2 .
	\end{equation*}
	
	If \(I\Subset(-mc^2,mc^2)\) is a compact interval, then there exists \(a_I>0\) such that, for every \(\lambda\in I\),
	\begin{equation}\label{eq2.7}
		\begin{split}
			\norm{u^+}_\G^2+\lambda\norm{u^+}_{L^2(\G)}^2
			&\ge a_I\norm{u^+}_\G^2,
			\qquad u^+\in\Y_\G^+,\\
			\norm{u^-}_\G^2-\lambda\norm{u^-}_{L^2(\G)}^2
			&\ge a_I\norm{u^-}_\G^2,
			\qquad u^-\in\Y_\G^- .
		\end{split}
	\end{equation}
	One may take
	\[
	a_I
	=
	1-\frac{\max_{\lambda\in I}\abs{\lambda}}{mc^2}>0,
	\]
	which is independent of the graph. The growth assumptions and the Sobolev embeddings imply that
	\[
	J_{\lambda,\G}\in C^1(\Y_\G,\R).
	\]
	Its derivative is
	\begin{equation*}
		J'_{\lambda,\G}(u)[\varphi]
		=
		\Ree\,\mathfrak d_\G(u,\varphi)
		+
		\lambda\Ree\int_\G u\cdot\overline{\varphi}\,\dd x
		-
		\Ree\int_\G f(\abs{u})u\cdot\overline{\varphi}\,\dd x,
	\end{equation*}
	where \(\mathfrak d_\G\) is the closed quadratic form of \(\D_\G\).
	
	The generalized Nehari manifold is
	\[
	\N_{\lambda,\G}
	=
	\left\{
	u\in\Y_\G\setminus\Y_\G^-:
	J'_{\lambda,\G}(u)[u]=0,\quad
	J'_{\lambda,\G}(u)[\eta]=0\ ,\quad \forall\,\eta\in\Y_\G^-
	\right\}.
	\]
	The ground level is
	\[
	d_\lambda(\G)
	=
	\inf_{u\in\N_{\lambda,\G}}J_{\lambda,\G}(u).
	\]
	
	For \(w\in\Y_\G^+\setminus\{0\}\), set
	\[
	\widehat{\Y}_\G(w)
	=
	\{tw+\eta:t\ge0,\ \eta\in\Y_\G^-\},
	\qquad
	\Gamma_{\lambda,\G}(w)
	=
	\max_{z\in\widehat{\Y}_\G(w)}
	J_{\lambda,\G}(z).
	\]
	
	\subsection{The generalized Nehari fibre formula}
	
	\begin{proposition}\label{Prop2.15}
		For every \(\lambda\in(-mc^2,mc^2)\),
		\[
		d_\lambda(\G)
		=
		\inf_{w\in\Y_\G^+\setminus\{0\}}
		\Gamma_{\lambda,\G}(w).
		\]
		The same formula holds on \(\R\) and on \(\R_+\) with the corresponding endpoint condition.
	\end{proposition}
	
	\begin{proof}
		We verify the standard generalized Nehari reduction of Szulkin and Weth \cite[Theorem 4.4]{SzulkinWeth2010} in the present setting.
		
		The splitting
		\[
		\Y_\G=\Y_\G^+\oplus\Y_\G^-
		\]
		is orthogonal both in the form norm and in \(L^2\).  Hence, for every fixed \(\lambda\in(-mc^2,mc^2)\), the quadratic part of \(J_{\lambda,\G}\) is positive definite on \(\Y_\G^+\) and negative definite on \(\Y_\G^-\).  More precisely, \eqref{eq2.7} holds on any compact subinterval \(I\Subset(-mc^2,mc^2)\) containing \(\lambda\).
		
		Let
		\[
		\Psi(u)=\int_\G F(\abs{u})\,\dd x.
		\]
		By the convexity of \(\xi\mapsto F(\abs{\xi})\), the functional \(\Psi\) is convex and weakly lower semicontinuous on \(\Y_\G\).  By the growth estimate \eqref{eq2.2} and the embeddings \(\Y_\G\hookrightarrow L^q(\G)\), it is of class \(C^1\).
		
		We next prove anti-coercivity on each fibre.  Fix \(w\in\Y_\G^+\setminus\{0\}\), and let
		\[
		z_n=t_nw+\eta_n\in\widehat{\Y}_\G(w),
		\qquad
		\eta_n\in\Y_\G^-,
		\]
		with \(t_n+\norm{\eta_n}_\G\to\infty\).  If \(t_n\) is bounded, then necessarily \(\norm{\eta_n}_\G\to\infty\), and the negative quadratic part gives
		\[
		J_{\lambda,\G}(z_n)\to-\infty,
		\]
		because \(\Psi\ge0\).
		
		Assume now that \(t_n\to\infty\).  If \(\norm{\eta_n}_\G/t_n\to\infty\), the same conclusion follows again from the negative quadratic term.  Otherwise, after passing to a subsequence,
		\[
		\frac{\eta_n}{t_n}\rightharpoonup\eta
		\quad\text{in }\Y_\G^- .
		\]
		Set
		\[
		y_n=w+\frac{\eta_n}{t_n},
		\qquad
		y=w+\eta.
		\]
		Then \(y_n\rightharpoonup y\) in \(\Y_\G\), and \(y^+=w\ne0\).  Hence \(y\ne0\).  Choose a compact subgraph \(K\subset\G\) and \(\rho>0\) such that
		\[
		\abs{\{x\in K:\abs{y(x)}>2\rho\}}>0.
		\]
		By the compact embedding \(\Y_\G\hookrightarrow L^2(K)\), a subsequence satisfies \(y_n\to y\) in \(L^2(K)\), hence in measure.  Therefore, for all large \(n\), \(\abs{y_n}\ge\rho\) on a subset of \(K\) whose measure is bounded below by a positive constant.  Using \eqref{eq2.4},
		\[
		\frac1{t_n^2}
		\int_\G F(t_n\abs{y_n})\,\dd x
		\to\infty .
		\]
		The positive quadratic contribution grows at most like \(Ct_n^2\), while the negative quadratic contribution is nonpositive.  Hence
		\[
		J_{\lambda,\G}(z_n)\to-\infty.
		\]
		Thus the restriction of \(J_{\lambda,\G}\) to each fibre is anti-coercive.
		
		On bounded fibre sets, \(J_{\lambda,\G}\) is weakly upper semicontinuous. Indeed, the positive part of the fibre is one-dimensional and therefore strongly continuous along weakly convergent sequences, while the negative quadratic part is weakly upper semicontinuous and \(-\Psi\) is weakly upper semicontinuous.  Therefore the maximum
		\[
		\Gamma_{\lambda,\G}(w)
		=
		\max_{\widehat{\Y}_\G(w)}J_{\lambda,\G}
		\]
		exists.
		
		For small \(s>0\), \eqref{eq2.2} gives
		\[
		J_{\lambda,\G}(sw)>0.
		\]
		On the other hand, \(J_{\lambda,\G}(\eta)\le0\) for every \(\eta\in\Y_\G^-\).  Hence the fibre maximum cannot occur at \(t=0\).  If
		\[
		z=tw+\eta
		\qquad(t>0,\ \eta\in\Y_\G^-)
		\]
		is a fibre maximizer, then the first variation vanishes in the directions \(w\) and \(\Y_\G^-\).  Since \(t>0\), this is equivalent to
		\[
		J'_{\lambda,\G}(z)[z]=0,
		\qquad
		J'_{\lambda,\G}(z)[\eta^-]=0
		\quad\forall\,\eta^-\in\Y_\G^-.
		\]
		Thus \(z\in\N_{\lambda,\G}\).  This proves
		\[
		\inf_{w\in\Y_\G^+\setminus\{0\}}\Gamma_{\lambda,\G}(w)
		\ge
		d_\lambda(\G).
		\]
		
		Conversely, let \(u\in\N_{\lambda,\G}\).  We prove that \(u\) is a global maximizer on its fibre.  A general element of \(\widehat{\Y}_\G(u^+)\) may be written as
		\[
		tu+v,
		\qquad
		t\ge0,\quad v\in\Y_\G^-,
		\]
		because
		\[
		tu+v
		=
		tu^++(tu^-+v).
		\]
		Let \(B_\lambda\) be the real symmetric bilinear form associated with \(Q_{\lambda,\G}\).  From the Nehari identities,
		\[
		B_\lambda(u,u)=\Psi'(u)[u],
		\qquad
		B_\lambda(u,v)=\Psi'(u)[v]
		\quad(v\in\Y_\G^-).
		\]
		Using the orthogonal splitting and the negativity of \(B_\lambda\) on \(\Y_\G^-\), one obtains
		\begin{align*}
			J_{\lambda,\G}(u)-J_{\lambda,\G}(tu+v)
			&\ge
			\Psi(tu+v)-\Psi(u)
			+\frac{1-t^2}{2}\Psi'(u)[u]
			-t\Psi'(u)[v].
		\end{align*}
		The omitted term is \(-\frac12B_\lambda(v,v)\ge0\).
		
		It remains to check a pointwise inequality.  Let \(\xi,\zeta\in\C^2\), \(t\ge0\), and put
		\[
		r=\abs{\xi},
		\qquad
		s=\abs{\zeta}.
		\]
		Since \(f\) is nondecreasing,
		\begin{equation}\label{eq2.8}
			F(s)-F(r)
			\ge
			\frac12 f(r)(s^2-r^2)
			\qquad(r,s\ge0).
		\end{equation}
		Moreover, by Cauchy's inequality,
		\[
		t\Ree(\xi\cdot\overline{\zeta})
		-
		\frac{t^2+1}{2}\abs{\xi}^2
		\le
		\frac12(\abs{\zeta}^2-\abs{\xi}^2).
		\]
		Combining this with \eqref{eq2.8}, we get
		\[
		F(\abs{\zeta})-F(\abs{\xi})
		+
		\frac{1-t^2}{2}f(\abs{\xi})\abs{\xi}^2
		-
		t f(\abs{\xi})
		\Ree\bigl(\xi\cdot\overline{\zeta-t\xi}\bigr)
		\ge0 .
		\]
		Apply this with
		\[
		\xi=u(x),
		\qquad
		\zeta=tu(x)+v(x),
		\]
		and integrate over \(\G\).  We obtain
		\[
		J_{\lambda,\G}(u)\ge J_{\lambda,\G}(tu+v)
		\qquad(t\ge0,\ v\in\Y_\G^-).
		\]
		Equivalently,
		\[
		J_{\lambda,\G}(u)
		\ge
		J_{\lambda,\G}(tu^++\eta)
		\qquad(t\ge0,\ \eta\in\Y_\G^-).
		\]
		Therefore \(u\) is a global maximizer on \(\widehat{\Y}_\G(u^+)\).  Taking the infimum over \(u\in\N_{\lambda,\G}\) gives
		\[
		d_\lambda(\G)
		\ge
		\inf_{w\in\Y_\G^+\setminus\{0\}}\Gamma_{\lambda,\G}(w).
		\]
		The two inequalities prove the formula.  The proofs on \(\R\) and on \(\R_+\), with the corresponding endpoint condition, are identical.
	\end{proof}
	
	\section{Long pendant profiles}\label{Sec3}
	
	Throughout this section, \(I\) denotes a fixed compact interval with \(I\Subset(-mc^2,mc^2)\).
	The proof of \cref{Thm1.1} is divided into several components.  We first recall the one-dimensional full-line result, then construct the truncated half-profile on a long pendant, and finally prove the convergence of the corresponding generalized Nehari fibres.
	
	\begin{proposition}\label{Prop3.1}
		Let \(I\) be a compact interval with \(I\Subset(-mc^2,mc^2)\).  For every \(\lambda\in I\), the full-line problem has a least-energy critical point \(U_\lambda\in\N_{\lambda,\R}\) such that
		\[
		J'_{\lambda,\R}(U_\lambda)=0,
		\qquad
		J_{\lambda,\R}(U_\lambda)=d_\lambda^\infty .
		\]
		It may be chosen with the parity
		\[
		U_{\lambda,1}\ \text{even},
		\qquad
		U_{\lambda,2}\ \text{odd}.
		\]
		Consequently \(U_{\lambda,2}(0)=0\).  Moreover, the chosen family is compact in \(H^1(\R,\C^2)\), and there exist \(C_I,\alpha_I>0\) such that
		\begin{equation}\label{eq3.1}
			\abs{U_\lambda(x)}+\abs{U_\lambda'(x)}
			\le C_Ie^{-\alpha_I\abs{x}}
			\qquad(x\in\R,\ \lambda\in I).
		\end{equation}
		The map \(\lambda\mapsto d_\lambda^\infty\) is continuous on \(I\), and
		\begin{equation*}
			d_I^*:=\inf_{\lambda\in I}d_\lambda^\infty>0 .
		\end{equation*}
	\end{proposition}
	
	\begin{proof}
		See \cref{AppA}.
	\end{proof}
	
	\begin{lemma}\label{Lem3.2}
		Let \(U_{\lambda,+}=U_\lambda|_{[0,\infty)}\).  Then \(U_{\lambda,+}\) satisfies the degree-one Dirac-Kirchhoff boundary condition at \(0\), and
		\[
		J_{\lambda,\R_+}(U_{\lambda,+})
		=\frac12d_\lambda^\infty .
		\]
		Moreover \(U_{\lambda,+}\in\N_{\lambda,\R_+}\), where \(\R_+\) is endowed with the terminal condition \(u_2(0)=0\).
	\end{lemma}
	
	\begin{proof}
		The boundary condition follows from the parity in \cref{Prop3.1}: \(U_{\lambda,2}(0)=0\).  The restriction solves the same Euler-Lagrange equation on \((0,\infty)\), and the boundary condition is exactly the natural degree-one Dirac-Kirchhoff condition.  Hence \(U_{\lambda,+}\) is a critical point of \(J_{\lambda,\R_+}\).
		
		Since \(U_\lambda\) is a critical point on \(\R\),
		\[
		J_{\lambda,\R}(U_\lambda)
		=
		J_{\lambda,\R}(U_\lambda)
		-\frac12J'_{\lambda,\R}(U_\lambda)[U_\lambda]
		=
		\int_\R H(\abs{U_\lambda})\,dx .
		\]
		The function \(H(\abs{U_\lambda})\) is even by the parity of \(U_\lambda\). Applying the same identity on the half-line gives
		\[
		J_{\lambda,\R_+}(U_{\lambda,+})
		=
		\int_0^\infty H(\abs{U_\lambda})\,dx
		=
		\frac12d_\lambda^\infty .
		\]
		
		It remains only to check that \(U_{\lambda,+}\notin\Y_{\R_+}^-\).  If \(U_{\lambda,+}\in\Y_{\R_+}^-\), then testing the critical equation with \(U_{\lambda,+}\) would give
		\[
		Q_{\lambda,\R_+}(U_{\lambda,+})
		=
		\int_0^\infty f(\abs{U_{\lambda,+}})\abs{U_{\lambda,+}}^2\,dx
		\ge0 .
		\]
		On the other hand, \(Q_{\lambda,\R_+}\) is strictly negative on \(\Y_{\R_+}^-\setminus\{0\}\), because \(\lambda\in(-mc^2,mc^2)\).  This is impossible.  Thus the positive spectral component of \(U_{\lambda,+}\) is nonzero, and the criticality identities imply \(U_{\lambda,+}\in\N_{\lambda,\R_+}\).
	\end{proof}
	
	\begin{lemma}\label{Lem3.3}
		Let \(\Pp\simeq[0,L]\) be a pendant path and place \(0\) at the terminal endpoint.  Let \(\chi_L\in C^\infty([0,L])\) satisfy
		\[
		\chi_L=1\text{ on }[0,L/2],
		\qquad
		\chi_L=0\text{ near }L,
		\qquad
		\abs{\chi_L'}\le C/L .
		\]
		Define
		\[
		W_{\lambda,L}(t)=\chi_L(t)U_{\lambda,+}(t)
		\]
		on \(\Pp\), using the concatenated coordinate described in \cref{Sec2}, and extend it by zero on \(\G\setminus\Pp\).  Then \(W_{\lambda,L}\in\dom(\D_\G)\), and
		\begin{equation*}
			\sup_{\lambda\in I}
			\left|
			J_{\lambda,\G}(W_{\lambda,L})-\frac12d_\lambda^\infty
			\right|
			\to0
			\qquad(L\to\infty).
		\end{equation*}
		Moreover,
		\begin{equation}\label{eq3.2}
			\sup_{\lambda\in I}
			\norm{W_{\lambda,L}^{\rm ext}-U_{\lambda,+}}_{H^1(\R_+)}
			\to0,
		\end{equation}
		where \(W_{\lambda,L}^{\rm ext}\) denotes the pendant function extended by zero to \(\R_+\).
	\end{lemma}
	
	\begin{proof}
		The function satisfies the terminal condition at \(0\), since \(U_{\lambda,2}(0)=0\).  At every suppressed degree-two vertex in the pendant chain, it is an \(H^1\)-function of the concatenated coordinate; hence it satisfies the transmission condition described after \cref{Def2.4}. It vanishes in a neighborhood of the attachment point \(L\).  Hence its zero extension is \(H^1\) across \(L\) when \(L\) lies in the interior of an edge, and it satisfies the vertex conditions when \(L\) is a vertex.  Thus \(W_{\lambda,L}\in\dom(\D_\G)\).
		
		For functions supported on the pendant, satisfying the terminal condition and vanishing near the attachment point, the closed form of \(\D_\G+\lambda\) agrees with the corresponding half-line form. Hence
		\[
		J_{\lambda,\G}(W_{\lambda,L})
		=
		J_{\lambda,\R_+}(\chi_LU_{\lambda,+}).
		\]
		By \eqref{eq3.1}, uniformly in \(\lambda\in I\),
		\[
		\norm{(1-\chi_L)U_{\lambda,+}}_{L^2(\R_+)}
		+
		\norm{(1-\chi_L)U_{\lambda,+}}_{L^p(\R_+)}
		+
		\norm{(\chi_LU_{\lambda,+})'-U_{\lambda,+}'}_{L^2(\R_+)}
		\le
		Ce^{-\alpha L/2}.
		\]
		The derivative estimate includes the term \(\chi_L'U_{\lambda,+}\), which is supported in \([L/2,L]\) and is bounded by \(CL^{-1}e^{-\alpha L/2}\).  Thus \eqref{eq3.2} follows.  The continuity of the quadratic form and the growth estimate \eqref{eq2.2} imply
		\[
		J_{\lambda,\R_+}(\chi_LU_{\lambda,+})
		\to
		J_{\lambda,\R_+}(U_{\lambda,+})
		\]
		uniformly on \(I\).  The conclusion follows from \cref{Lem3.2}.
	\end{proof}
	
	We next introduce the localization tools used to compare a long pendant with the endpoint half-line.  All constants and moduli in these tools are independent of the generalized star graph and of the geometry of its compact core.  In \cref{Lem3.5,Lem3.6} they depend only on \(m,c\), the fixed cutoff functions, and the displayed parameters \(R,\Omega\), or the compact half-line set \(K\).  In the fibre lemmas they may additionally depend on \(I,f\), the compact family of half-line profiles, and the convergence modulus in \eqref{eq3.12}.  For \(v_\lambda=U_{\lambda,+}\), these quantities depend only on \(I,m,c,f\). None depends on the shortest bounded edge, the vertex degrees, the number of edges, or the topology of the remaining compact core.
	
	\begin{lemma}\label{Lem3.4}
		Let \(\G\) contain a pendant path \(\Pp\simeq[0,L]\), with \(0\) the terminal endpoint and \(L\) the attachment endpoint.  Let \(\theta_L\in C^\infty([0,L])\) satisfy
		\[
		0\le \theta_L\le1,\qquad
		\theta_L=1\text{ on }[0,L/2],\qquad
		\theta_L=0\text{ near }L,\qquad
		|\theta_L'|\le C/L .
		\]
		Define
		\[
		E_L:\Y_{\R_+}\to\Y_\G,
		\qquad
		(E_Lh)(t)=\theta_L(t)h(t)\quad\text{on }\Pp,
		\qquad
		E_Lh=0\quad\text{on }\G\setminus\Pp ,
		\]
		and
		\[
		R_L:\Y_\G\to\Y_{\R_+},
		\qquad
		(R_Lu)(t)=\theta_L(t)u_{\Pp}(t)\quad(0\le t\le L),
		\qquad
		R_Lu=0\quad(t\ge L).
		\]
		Then \(E_L\) and \(R_L\) are bounded uniformly for \(L\ge1\) and in \(\G\).  Moreover,
		\begin{equation}\label{eq3.3}
			R_LE_Lh=\theta_L^2h\to h
			\qquad\text{in }\Y_{\R_+}
		\end{equation}
		for every \(h\in\Y_{\R_+}\), and the convergence is uniform for \(h\) in compact subsets of \(\Y_{\R_+}\).
		
		Let \(Q_{\lambda,\G}\) be the closed quadratic form associated with \(\D_\G+\lambda\), and let \(Q_{\lambda,+}\) be the corresponding half-line form. Then, for every compact set \(K\subset\Y_{\R_+}\),
		\begin{equation}\label{eq3.4}
			\sup_{\lambda\in I}\sup_{h\in K}
			\left|
			Q_{\lambda,\G}(E_Lh)-Q_{\lambda,+}(h)
			\right|
			\to0 .
		\end{equation}
		Moreover, if \(h\in\Y_{\R_+}\) is supported in \([0,R]\) and \(L>4R\), then
		\begin{equation}\label{eq3.5}
			Q_{\lambda,\G}(E_Lh)=Q_{\lambda,+}(h)
			\qquad(\lambda\in I).
		\end{equation}
	\end{lemma}
	
	\begin{proof}
		The form domain satisfies
		\[
		\Y_\G=\dom(|\D_\G|^{1/2})
		=
		[L^2(\G),\dom(\D_\G)]_{1/2}
		\]
		with equivalent norm, and similarly on \(\R_+\).  Multiplication by a bounded Lipschitz function continuous at vertices is bounded on \(L^2\), preserves the Dirac-Kirchhoff vertex conditions, and is bounded on \(\dom(\D)\), with norm controlled by the \(L^\infty\)-norm and the Lipschitz constant of the multiplier. For \(E_L\), the zero extension creates no mismatch at the attachment point because \(\theta_L=0\) near \(L\): it is \(H^1\) across \(L\) if this point is in the interior of an edge, and it satisfies the vertex conditions if \(L\) is a vertex.  The terminal condition at \(0\) is the same as on the half-line. For \(R_L\), the factor \(\theta_L\) makes the extension by zero past \(L\) admissible in either case.  For \(L\ge1\), one has \(\|\theta_L\|_\infty\le1\) and \(\|\theta_L'\|_\infty\le C\), so interpolation gives the asserted uniform boundedness.
		
		For \eqref{eq3.3}, first take \(h\in C_c^\infty(\R_+,\C^2)\) satisfying the terminal condition.  Then \(\theta_L^2h=h\) for all large \(L\).  The general case follows by density and uniform boundedness of the multipliers \(\theta_L^2\) on \(\Y_{\R_+}\).  Uniformity on compact subsets follows by a finite-net argument.
		
		If \(h\in\dom(\D_+)\) is supported in \([0,R]\) and \(L>4R\), then \(E_Lh\) is supported in the part of the pendant where the graph and half-line operators have the same differential expression and the same terminal condition.  Thus
		\[
		Q_{\lambda,\G}(E_Lh)=Q_{\lambda,+}(h).
		\]
		For \(h\in\Y_{\R_+}\) supported in \([0,R]\), approximate \(h\) in \(\Y_{\R_+}\) by \(\dom(\D_+)\)-functions satisfying the terminal condition and supported in a slightly larger fixed interval.  Since the forms are closed and \(E_L\) is uniformly bounded, the identity extends to every such \(h\).  The convergence \eqref{eq3.4} follows by approximating \(K\) by finitely many compactly supported functions and using the uniform boundedness of the forms on bounded form-norm sets.
	\end{proof}
	
	\begin{lemma}\label{Lem3.5}
		Let \(\G,\Pp,E_L,R_L\) be as in \cref{Lem3.4}.  Let \(P_\G^\pm\) and \(P_+^\pm\) be the positive and negative spectral projections of \(\D_\G\) and of the half-line operator \(\D_+\) on \(\R_+\) with terminal condition \(u_2(0)=0\).
		
		For every fixed \(R>0\) there exists a modulus \(\varepsilon_{L,R}\to0\), independent of the graph, such that if \(h\in\Y_{\R_+}\) is supported in \([0,R]\), then for \(L>4R\),
		\begin{equation}\label{eq3.6}
			\begin{split}
				&\norm{P_\G^\pm E_Lh-E_LP_+^\pm h}_{\Y_\G}
				+
				\norm{R_LP_\G^\pm E_Lh-P_+^\pm h}_{\Y_{\R_+}}  \\
				&\qquad\le
				\varepsilon_{L,R}\norm{h}_{\Y_{\R_+}} .
			\end{split}
		\end{equation}
		Consequently, for every compact set \(K\subset\Y_{\R_+}\),
		\begin{equation}\label{eq3.7}
			\sup_{h\in K}
			\norm{P_\G^\pm E_Lh-E_LP_+^\pm h}_{\Y_\G}
			\to0 ,
		\end{equation}
		and
		\begin{equation}\label{eq3.8}
			\sup_{h\in K}
			\norm{R_LP_\G^\pm E_Lh-P_+^\pm h}_{\Y_{\R_+}}
			\to0 .
		\end{equation}
		The convergence is uniform over all generalized star graphs containing such a pendant path.
		
		More generally, if \(\widetilde h_L(h)\in\Y_\G\), \(h\in K\), satisfies
		\[
		\sup_{h\in K}
		\norm{\widetilde h_L(h)-E_Lh}_{\Y_\G}\to0,
		\]
		then
		\begin{equation}\label{eq3.9}
			\sup_{h\in K}
			\norm{P_\G^\pm\widetilde h_L(h)-E_LP_+^\pm h}_{\Y_\G}
			\to0 .
		\end{equation}
	\end{lemma}
	
	\begin{proof}
		Put \(\omega=mc^2\).  Let
		\[
		\Omega\Subset
		\{z\in\C:\operatorname{dist}(z,(-\infty,-\omega]\cup[\omega,\infty))>0\}.
		\]
		For every graph considered here, \(\Omega\subset\rho(\D_\G)\), and the same is true for \(\D_+\).
		
		We first gived a graph-uniform Combes-Thomas estimate.  Let \(\varphi\) be a bounded real Lipschitz function on \(\G\), continuous at vertices and satisfying \(\abs{\varphi'}\le1\). Multiplication by \(e^{\pm\alpha\varphi}\) preserves \(\dom(\D_\G)\), and
		\[
		e^{\alpha\varphi}\D_\G e^{-\alpha\varphi}
		=
		\D_\G+ic\alpha\sigma_1\varphi'
		\]
		as operators from \(\dom(\D_\G)\) to \(L^2(\G)\).  For \(\alpha>0\) small enough, depending only on \(m,c\) and \(\Omega\), the perturbation has norm smaller than one half of the distance from \(\Omega\) to the common spectrum.  Thus
		\[
		\norm{e^{\alpha\varphi}(\D_\G-z)^{-1}e^{-\alpha\varphi}}_{L^2\to L^2}
		\le C_\Omega
		\qquad(z\in\Omega),
		\]
		with \(C_\Omega\) independent of the graph and of \(\|\varphi\|_\infty\). For nonempty closed sets \(A,B\subset\G\), choose the bounded weight
		\[
		\varphi(x)=\min\{\dist(x,A),\dist(A,B)\}.
		\]
		It vanishes on \(A\) and equals \(\dist(A,B)\) on \(B\). Since
		\[
		\D_\G(\D_\G-z)^{-1}=I+z(\D_\G-z)^{-1},
		\]
		the same estimate gives
		\[
		\norm{\mathbf 1_B(\D_\G-z)^{-1}\mathbf 1_A}_{L^2\to L^2}
		\le C_\Omega e^{-\alpha\dist(A,B)}
		\]
		for closed sets \(A,B\subset\G\).  If \(\dist(A,B)>0\), the resolvent equation and the first-order differential expression also give the local estimate
		\[
		\norm{\big((\D_\G-z)^{-1}\mathbf 1_Ah\big)|_B}_{H^1(B)}
		\le C_\Omega e^{-\alpha\dist(A,B)}\norm{h}_{L^2(\G)},
		\]
		where \(H^1(B)\) denotes the sum of the edgewise local \(H^1\)-norms.  Thus no sharp characteristic cutoff is asserted to preserve \(\dom(\D_\G)\).
		
		Let \(h\in C_c^\infty([0,R],\C^2)\) satisfy the terminal condition and assume \(L>4R\).  Set
		\[
		u_+(z):=(\D_+-z)^{-1}h .
		\]
		The half-line version of the preceding estimate gives
		\[
		\norm{u_+(z)}_{H^1([s,\infty))}
		\le C_\Omega e^{-\alpha(s-R)}
		\norm{h}_{L^2(\R_+)}
		\qquad(s\ge R).
		\]
		Use the cutoff \(\theta_L\) introduced in \cref{Lem3.4}.  Since \(L>4R\), one has \(\theta_L=1\) on \(\supp h\).  Extending the commutator term by zero outside \(\Pp\), the product rule gives
		\[
		(\D_\G-z)E_Lu_+(z)
		=
		E_Lh-ic\sigma_1\theta_L'u_+(z).
		\]
		The commutator term is supported in \([L/2,L]\), and therefore
		\[
		\norm{\theta_L'u_+(z)}_{L^2}
		\le
		C_\Omega e^{-\alpha(L/2-R)}
		\norm{h}_{L^2}.
		\]
		It follows that
		\[
		\norm{(\D_\G-z)^{-1}E_Lh
			-
			E_L(\D_+-z)^{-1}h}_{\Y_\G}
		\le
		C_{\Omega,R}e^{-\alpha L/4}\norm{h}_{\Y_{\R_+}},
		\qquad L>4R .
		\]
		By density and the uniform resolvent bounds, this estimate holds for every \(h\in\Y_{\R_+}\) supported in \([0,R]\).
		
		To pass from resolvents to spectral projections without using a compactly supported spectral cutoff, we use the sign-resolvent formula.  If \(D\) is self-adjoint and \(\sigma(D)\cap(-\omega,\omega)=\varnothing\), then
		\[
		\begin{split}
			P_D^\pm
			&=
			\frac12I
			\mathbin{\pm}\frac1\pi\,
			\mathop{\mathrm{s}\!-\!\lim}_{T\to\infty}
			\int_0^T D(D^2+s^2)^{-1}\,\dd s\\
			&=
			\frac12I
			\mathbin{\pm}\frac1{2\pi}\,
			\mathop{\mathrm{s}\!-\!\lim}_{T\to\infty}
			\int_0^T
			\bigl((D-is)^{-1}+(D+is)^{-1}\bigr)\,\dd s.
		\end{split}
		\]
		The identity terms cancel in \(P_\G^\pm E_L-E_LP_+^\pm\).  For \(z=\pm is\), the product-rule identity above and the resolvent identity give
		\[
		(\D_\G-z)^{-1}E_L-E_L(\D_+-z)^{-1}
		=
		ic(\D_\G-z)^{-1}\sigma_1\theta_L'
		(\D_+-z)^{-1},
		\]
		on functions supported in \([0,R]\).  Keeping the dependence on \(s\) in the preceding Combes--Thomas argument and using the common gap gives the two elementary bounds
		\[
		\norm{(D\mp is)^{-1}}_{L^2\to\Y_D}
		\le C(1+s)^{-1/2}
		\]
		and, for separated measurable sets \(A,B\),
		\[
		\norm{\mathbf 1_B(D\mp is)^{-1}\mathbf 1_A}_{L^2\to L^2}
		\le C(1+s)^{-1}e^{-\alpha\dist(A,B)}.
		\]
		Here \(\Y_D=\dom(|D|^{1/2})\) with its graph norm.  The first estimate follows directly from the spectral theorem; the second follows from the preceding weighted-conjugation argument with the spectral-distance dependence retained. Applying the first estimate to the left resolvent in the product-rule identity and the second to the right resolvent, across the support of \(\theta_L'\), yields
		\[
		\begin{split}
			&\norm{(\D_\G\mp is)^{-1}E_Lh
				-E_L(\D_+\mp is)^{-1}h}_{\Y_\G}\\
			&\qquad\le
			C_R e^{-\alpha L/4}(1+s)^{-3/2}
			\norm{h}_{\Y_{\R_+}},
			\qquad s\ge0,\quad L>4R,
		\end{split}
		\]
		with constants independent of the graph.  The right-hand side is integrable in \(s\).  Hence the sign-resolvent representation and dominated convergence give
		\[
		\norm{P_\G^\pm E_Lh-E_LP_+^\pm h}_{\Y_\G}
		\le
		\varepsilon_{L,R}\norm{h}_{\Y_{\R_+}},
		\qquad
		\varepsilon_{L,R}\to0,
		\]
		for \(h\) supported in \([0,R]\).
		
		It remains to estimate the restriction term.  We write
		\[
		\begin{split}
			R_LP_\G^\pm E_Lh-P_+^\pm h
			&=
			R_L\bigl(P_\G^\pm E_Lh-E_LP_+^\pm h\bigr)\\
			&\quad+
			(R_LE_L-I)P_+^\pm h .
		\end{split}
		\]
		The first term is controlled by the preceding estimate and the uniform boundedness of \(R_L\).  For the second term, choose a smooth half-line cutoff \(\zeta_L\) such that
		\[
		\zeta_L=0\text{ on }[0,R],\qquad
		\zeta_L=1\text{ on }[L/2,\infty),\qquad
		\norm{\zeta_L'}_\infty\le C/L .
		\]
		Since \(\zeta_Lh=0\) and \(1-\theta_L^2\) is supported where \(\zeta_L=1\),
		\[
		(1-\theta_L^2)P_+^\pm h
		=(1-\theta_L^2)[\zeta_L,P_+^\pm]h .
		\]
		The sign-resolvent representation and the spectral theorem give
		\[
		\norm{(D\mp is)^{-1}}_{L^2\to\Y_D}
		\le C(1+s)^{-1/2},
		\qquad
		\norm{(D\mp is)^{-1}}_{\Y_D\to L^2}
		\le C(1+s)^{-1}.
		\]
		Thus the two-resolvent formula for \([\zeta_L,P_+^\pm]\) has the integrable bound \(C(1+s)^{-3/2}\norm{\zeta_L'}_\infty\), and hence
		\[
		\norm{(R_LE_L-I)P_+^\pm h}_{\Y_{\R_+}}
		\le \frac{C}{L}\norm{h}_{\Y_{\R_+}}.
		\]
		This proves \eqref{eq3.6}.
		
		For a compact set \(K\subset\Y_{\R_+}\), choose \(R\) and finitely many \(h_1,\ldots,h_M\), each supported in \([0,R]\), such that \(K\) is covered by small \(\Y_{\R_+}\)-balls around the \(h_j\).  The uniform boundedness of \(E_L,R_L,P_\G^\pm,P_+^\pm\), together with the supported estimate, gives \eqref{eq3.7} and \eqref{eq3.8}.  Finally, \eqref{eq3.9} follows from
		\[
		\begin{split}
			\norm{P_\G^\pm\widetilde h_L(h)-E_LP_+^\pm h}_{\Y_\G}
			&\le
			\norm{\widetilde h_L(h)-E_Lh}_{\Y_\G}\\
			&\quad+
			\norm{P_\G^\pm E_Lh-E_LP_+^\pm h}_{\Y_\G}.
		\end{split}
		\]
	\end{proof}
	
	\begin{lemma}\label{Lem3.6}
		Let \(\zeta_R\) be a real Lipschitz cutoff on \(\G\), continuous at vertices and satisfying
		\[
		\norm{\zeta_R'}_{L^\infty(\G)}\le \frac{C}{R}.
		\]
		Then
		\begin{equation}\label{eq3.10}
			\norm{[P_\G^\pm,\zeta_R]u}_{\Y_\G}
			\le \frac{C}{R}\norm{u}_{\Y_\G}
			\qquad(u\in\Y_\G),
		\end{equation}
		where \(C\) is independent of the graph.
		
		If \(\chi_R,\psi_R\) are real Lipschitz functions, continuous at vertices, and
		\[
		\chi_R^2+\psi_R^2=1,
		\qquad
		\norm{\chi_R'}_{L^\infty}+\norm{\psi_R'}_{L^\infty}\le \frac{C}{R},
		\]
		then, for every \(\lambda\in I\),
		\begin{equation}\label{eq3.11}
			Q_{\lambda,\G}(u)
			=
			Q_{\lambda,\G}(\chi_Ru)
			+
			Q_{\lambda,\G}(\psi_Ru)
			\qquad(u\in\Y_\G).
		\end{equation}
	\end{lemma}
	
	\begin{proof}
		For \(u\in\dom(\D_\G)\),
		\[
		[\D_\G,\zeta_R]u=-ic\sigma_1\zeta_R'u .
		\]
		The identity term in the sign-resolvent representation used in the proof of \cref{Lem3.5} commutes with \(\zeta_R\), while the resolvent identity gives
		\[
		[P_\G^\pm,\zeta_R]
		=
		\mp\frac1{2\pi}
		\int_0^\infty
		\sum_{z\in\{is,-is\}}
		(\D_\G-z)^{-1}[\D_\G,\zeta_R]
		(\D_\G-z)^{-1}\,\dd s .
		\]
		The spectral theorem and the common gap give
		\[
		\norm{(\D_\G\mp is)^{-1}}_{L^2\to\Y_\G}
		\le C(1+s)^{-1/2},
		\qquad
		\norm{(\D_\G\mp is)^{-1}}_{\Y_\G\to L^2}
		\le C(1+s)^{-1}.
		\]
		Hence the two-resolvent integrand, as an operator on \(\Y_\G\), is bounded by \(C(1+s)^{-3/2}\norm{\zeta_R'}_\infty\), which is integrable and independent of the graph.  Since
		\[
		\norm{[\D_\G,\zeta_R]}_{L^2\to L^2}
		\le C\norm{\zeta_R'}_{L^\infty}
		\le C/R,
		\]
		we obtain \eqref{eq3.10}.
		
		To prove \eqref{eq3.11}, first take \(u\in\dom(\D_\G)\). Then \(\chi_Ru,\psi_Ru\in\dom(\D_\G)\).  Since
		\[
		\chi_R\chi_R'+\psi_R\psi_R'=0
		\]
		on every edge, expanding the first-order Dirac form gives
		\[
		\mathfrak d_\G(u,u)
		=
		\mathfrak d_\G(\chi_Ru,\chi_Ru)
		+
		\mathfrak d_\G(\psi_Ru,\psi_Ru).
		\]
		The \(L^2\)-term is exactly local:
		\[
		\norm{u}_{L^2(\G)}^2
		=
		\norm{\chi_Ru}_{L^2(\G)}^2+
		\norm{\psi_Ru}_{L^2(\G)}^2 .
		\]
		Thus \eqref{eq3.11} holds on \(\dom(\D_\G)\).  Since \(\dom(\D_\G)\) is a form core for \(\Y_\G\), and multiplication by \(\chi_R,\psi_R\) is bounded on \(\Y_\G\), the identity extends to all \(u\in\Y_\G\).
	\end{proof}
	
	\begin{lemma}\label{Lem3.7}
		Let \(\lambda_n\to\lambda\) in \(I\), and let \(w_n,w\in\Y_{\R_+}^+\) satisfy
		\[
		w_n\to w\quad\text{in }\Y_{\R_+},
		\qquad
		w\ne0 .
		\]
		Then
		\[
		\Gamma_{\lambda_n,\R_+}(w_n)
		\to
		\Gamma_{\lambda,\R_+}(w).
		\]
	\end{lemma}
	
	\begin{proof}
		The vectors \(w_n\) stay in a compact subset of \(\Y_{\R_+}^+\setminus\{0\}\), and \(\lambda_n\) stays in a compact subinterval of the spectral gap.  The fibre anti-coercivity argument in \cref{Prop2.15}, with constants uniform on this compact set, implies that all fibre maximizers over \(\widehat\Y_{\R_+}(w_n)\) are uniformly bounded.
		
		Let
		\[
		z_n=t_nw_n+\eta_n,
		\qquad
		t_n\ge0,\quad \eta_n\in\Y_{\R_+}^-,
		\]
		be maximizers.  Passing to a subsequence, we may assume that \(t_n\to t\ge0\) and \(\eta_n\rightharpoonup\eta\) in \(\Y_{\R_+}^-\).  The positive part \(t_nw_n\) converges strongly to \(tw\).  The negative quadratic part is weakly upper semicontinuous, and \(-\int_{\R_+}F(\abs{\cdot})\) is weakly upper semicontinuous because \(u\mapsto\int_{\R_+}F(\abs{u})\) is convex and weakly lower semicontinuous.  Hence
		\[
		\limsup_{n\to\infty}
		\Gamma_{\lambda_n,\R_+}(w_n)
		\le
		J_{\lambda,\R_+}(tw+\eta)
		\le
		\Gamma_{\lambda,\R_+}(w).
		\]
		
		Conversely, let \(z=tw+\eta\) be a maximizer on \(\widehat\Y_{\R_+}(w)\), with \(\eta\in\Y_{\R_+}^-\).  Then
		\[
		z_n^*:=tw_n+\eta\in\widehat\Y_{\R_+}(w_n),
		\qquad
		z_n^*\to z\quad\text{in }\Y_{\R_+}.
		\]
		Therefore
		\[
		\liminf_{n\to\infty}
		\Gamma_{\lambda_n,\R_+}(w_n)
		\ge
		\lim_{n\to\infty}J_{\lambda_n,\R_+}(z_n^*)
		=
		J_{\lambda,\R_+}(z)
		=
		\Gamma_{\lambda,\R_+}(w).
		\]
		The two inequalities prove the claim.
	\end{proof}
	
	\begin{lemma}\label{Lem3.8}
		Let \(I\) be a compact interval with \(I\Subset(-mc^2,mc^2)\).  For each \(L>0\), let \(\G_L\) be an arbitrary generalized star graph carrying a distinguished pendant path isometric to \([0,L]\), with terminal endpoint \(0\).  Let \(v_{\lambda,L}\in H^1([0,L],\C^2)\) vanish near \(L\) and satisfy the terminal condition at \(0\). We use the same notation for its zero extension to \(\G_L\), and write \(v_{\lambda,L}^{\rm ext}\) for its zero extension to \(\R_+\). Suppose that there is a family \(v_\lambda\in H^1(\R_+,\C^2)\) satisfying the terminal condition such that
		\begin{equation}\label{eq3.12}
			\sup_{\lambda\in I}
			\norm{v_{\lambda,L}^{\rm ext}-v_\lambda}_{H^1(\R_+)}
			\to0
			\qquad(L\to\infty).
		\end{equation}
		Assume also that \(\{v_\lambda:\lambda\in I\}\) is compact in \(H^1(\R_+,\C^2)\), that \(P_+^+v_\lambda\ne0\) for every \(\lambda\in I\), and that \(v_\lambda\) is a maximizer on its half-line fibre, namely
		\[
		J_{\lambda,\R_+}(v_\lambda)
		=
		\Gamma_{\lambda,\R_+}(P_+^+v_\lambda)
		\qquad(\lambda\in I).
		\]
		Then
		\begin{equation}\label{eq3.13}
			\sup_{\lambda\in I}
			\left|
			\Gamma_{\lambda,\G_L}\big(P_{\G_L}^+v_{\lambda,L}\big)
			-
			\Gamma_{\lambda,\R_+}\big(P_+^+v_\lambda\big)
			\right|
			\to0 .
		\end{equation}
		The convergence is uniform with respect to the choice of \(\G_L\) and of the distinguished pendant path.
	\end{lemma}
	
	\begin{proof}
		Let \(E_L\) and \(R_L\) be as in \cref{Lem3.4}.  Write
		\[
		P_L^\pm=P_{\G_L}^\pm,
		\qquad
		\Y_L=\Y_{\G_L},
		\qquad
		J_{\lambda,L}=J_{\lambda,\G_L}.
		\]
		Set
		\[
		w_{\lambda,L}:=P_L^+v_{\lambda,L},
		\qquad
		w_\lambda:=P_+^+v_\lambda .
		\]
		By \eqref{eq3.12}, compactness of \(\{v_\lambda:\lambda\in I\}\), and \cref{Lem3.5},
		\begin{equation}\label{eq3.14}
			\sup_{\lambda\in I}
			\norm{w_{\lambda,L}-E_Lw_\lambda}_{\Y_L}
			\to0,
			\qquad
			\sup_{\lambda\in I}
			\norm{R_Lw_{\lambda,L}-w_\lambda}_{\Y_{\R_+}}
			\to0 .
		\end{equation}
		Since \(\{w_\lambda:\lambda\in I\}\) is compact in \(\Y_{\R_+}\) and does not contain \(0\), there is \(\delta>0\) such that
		\begin{equation}\label{eq3.15}
			\norm{w_\lambda}_{\Y_{\R_+}}\ge\delta
			\qquad(\lambda\in I).
		\end{equation}
		Thus \(w_{\lambda,L}\ne0\) for all large \(L\), uniformly in \(\lambda\).
		
		It is enough to prove sequential convergence.  Let \(L_n\to\infty\) and \(\lambda_n\in I\).  Passing to a subsequence, assume
		\[
		\lambda_n\to\lambda_*\in I,
		\qquad
		v_{\lambda_n}\to v_*
		\quad\text{in }H^1(\R_+,\C^2).
		\]
		Put
		\[
		\omega_n=P_+^+v_{\lambda_n},
		\qquad
		\omega_*=P_+^+v_*.
		\]
		Then \(\omega_n\to\omega_*\) in \(\Y_{\R_+}\), and \(\omega_*\neq 0\).  By \eqref{eq3.14},
		\begin{equation}\label{eq3.16}
			w_n:=w_{\lambda_n,L_n}
			=
			E_{L_n}\omega_n+o(1)
			\quad\text{in }\Y_{\G_{L_n}},
			\qquad
			R_{L_n}w_n\to\omega_*
			\quad\text{in }\Y_{\R_+}.
		\end{equation}
		We prove
		\begin{equation}\label{eq3.17}
			\Gamma_{\lambda_n,\G_{L_n}}(w_n)
			-
			\Gamma_{\lambda_n,\R_+}(\omega_n)
			\to0 .
		\end{equation}
		By \cref{Lem3.7},
		\begin{equation}\label{eq3.18}
			\Gamma_{\lambda_n,\R_+}(\omega_n)
			\to
			\Gamma_{\lambda_*,\R_+}(\omega_*).
		\end{equation}
		
		In the remainder of the sequential argument, \(o_n(1)\) denotes a quantity tending to zero as \(n\to\infty\) with \(R\) fixed, whereas \(o_R(1)\) denotes a quantity satisfying
		\[
		\lim_{R\to\infty}\limsup_{n\to\infty}\abs{o_R(1)}=0.
		\]
		Thus \(n\to\infty\) is always taken before \(R\to\infty\), and all constants are independent of the graphs \(\G_{L_n}\) and of their compact-core decorations.
		
		Let
		\[
		z_n=t_nw_n+\eta_n,
		\qquad
		t_n\ge0,\quad \eta_n\in\Y_{\G_{L_n}}^-,
		\]
		be a maximizer of \(J_{\lambda_n,\G_{L_n}}\) on \(\widehat\Y_{\G_{L_n}}(w_n)\).
		
		We first prove that the maximizers are uniformly bounded:
		\begin{equation}\label{eq3.19}
			0\le t_n\le C,
			\qquad
			\norm{\eta_n}_{\G_{L_n}}\le C .
		\end{equation}
		By the fibre-maximizing assumption, the compactness of the half-line profile family, and \eqref{eq3.15}, one has
		\[
		\gamma_I
		:=
		\inf_{\lambda\in I}
		\Gamma_{\lambda,\R_+}(P_+^+v_\lambda)
		>0.
		\]
		Indeed, on the fixed half-line the compact family \(\{P_+^+v_\lambda:\lambda\in I\}\) is bounded away from zero; the uniform spectral coercivity and the fixed half-line Sobolev embedding therefore give a common positive lower bound by testing a sufficiently small multiple of each positive direction. Moreover, \eqref{eq3.12} and the locality of the quadratic and nonlinear terms give
		\[
		\sup_{\lambda\in I}
		\left|
		J_{\lambda,\G_L}(v_{\lambda,L})
		-
		J_{\lambda,\R_+}(v_\lambda)
		\right|
		\to0
		\qquad(L\to\infty),
		\]
		uniformly with respect to the graph.  Since
		\[
		v_{\lambda_n,L_n}
		=
		w_n+P_{\G_{L_n}}^-v_{\lambda_n,L_n}
		\in\widehat\Y_{\G_{L_n}}(w_n),
		\]
		we obtain, for all sufficiently large \(n\),
		\begin{equation}\label{eq3.20}
			\begin{split}
				J_{\lambda_n,\G_{L_n}}(z_n)
				&=
				\Gamma_{\lambda_n,\G_{L_n}}(w_n)\\
				&\ge
				J_{\lambda_n,\G_{L_n}}(v_{\lambda_n,L_n})
				\ge \frac12\gamma_I=: \mu>0 .
			\end{split}
		\end{equation}
		
		Assume by contradiction that \(\norm{z_n}_{\G_{L_n}}\to\infty\).  Set
		\[
		r_n:=\norm{z_n}_{\G_{L_n}},
		\qquad
		y_n:=\frac{z_n}{r_n},
		\qquad
		a_n:=\frac{t_n}{r_n},
		\qquad
		\xi_n:=\frac{\eta_n}{r_n}.
		\]
		The spectral splitting gives
		\[
		1=a_n^2\norm{w_n}_{\G_{L_n}}^2+\norm{\xi_n}_{\G_{L_n}}^2 .
		\]
		If \(a_n\to0\), then \(\norm{\xi_n}_{\G_{L_n}}\to1\).  By the uniform spectral-gap coercivity,
		\[
		Q_{\lambda_n,\G_{L_n}}(y_n)
		\le
		Ca_n^2-a_I\norm{\xi_n}_{\G_{L_n}}^2
		\le
		-\frac{a_I}{2}
		\]
		for all large \(n\).  Since \(F\ge0\), this gives
		\[
		J_{\lambda_n,\G_{L_n}}(z_n)
		\le
		-\frac{a_I}{4}r_n^2
		\to-\infty,
		\]
		contradicting \eqref{eq3.20}.  Therefore, after passing to a subsequence,
		\begin{equation}\label{eq3.21}
			a_n\to a>0 .
		\end{equation}
		
		Define
		\[
		\bar\xi_n:=R_{L_n}\xi_n\in\Y_{\R_+}.
		\]
		The sequence \(\bar\xi_n\) is bounded in \(\Y_{\R_+}\).  Passing to a subsequence,
		\[
		\bar\xi_n\rightharpoonup \xi
		\quad\text{in }\Y_{\R_+},
		\qquad
		\bar\xi_n\to\xi
		\quad\text{in }L^q_{\rm loc}(\R_+)
		\]
		for every \(2\le q<\infty\).  We claim that \(\xi\in\Y_{\R_+}^-\).  Let \(\phi\in\Y_{\R_+}\).  Since \(\xi_n\in\Y_{\G_{L_n}}^-\),
		\[
		\bigl(\xi_n,P_{\G_{L_n}}^+E_{L_n}\phi\bigr)_{L^2(\G_{L_n})}=0 .
		\]
		By \cref{Lem3.5},
		\[
		P_{\G_{L_n}}^+E_{L_n}\phi
		=
		E_{L_n}P_+^+\phi+o(1)
		\qquad\text{in }\Y_{\G_{L_n}},
		\]
		and hence in \(L^2\).  Therefore
		\[
		\bigl(\xi_n,E_{L_n}P_+^+\phi\bigr)_{L^2(\G_{L_n})}\to0 .
		\]
		For every \(g\in L^2(\R_+,\C^2)\), the definitions of \(E_L\) and \(R_L\) give
		\[
		\bigl(\xi_n,E_{L_n}g\bigr)_{L^2(\G_{L_n})}
		=
		\bigl(R_{L_n}\xi_n,g\bigr)_{L^2(\R_+)} .
		\]
		Taking \(g=P_+^+\phi\) and passing to the weak \(L^2\)-limit gives
		\[
		(\xi,P_+^+\phi)_{L^2(\R_+)}=0 .
		\]
		Since \(\phi\) is arbitrary, \(\xi\perp_{L^2}\Y_{\R_+}^+\), and hence \(\xi\in\Y_{\R_+}^-\).
		
		By \eqref{eq3.16} and \eqref{eq3.21},
		\[
		R_{L_n}y_n
		=
		a_nR_{L_n}w_n+\bar\xi_n
		\rightharpoonup
		y:=a\omega_*+\xi
		\quad\text{in }\Y_{\R_+}.
		\]
		Since \(P_+^+\xi=0\) and \(a\omega_*\neq0\), one has \(y\ne0\).  Hence there exist \(R_0>0\), \(\rho>0\), and a measurable set \(A\subset[0,R_0]\) of positive measure such that \(\abs{y}>2\rho\) on \(A\).  Strong local \(L^2\)-convergence gives, after passing to a subset of \(A\) of positive measure,
		\[
		\abs{R_{L_n}y_n}\ge\rho
		\qquad\text{on }A
		\]
		for all large \(n\).  On \(A\), \(R_{L_n}y_n\) agrees with the pendant representative of \(y_n\).  Therefore, by \eqref{eq2.4},
		\[
		\frac1{r_n^2}
		\int_{\G_{L_n}}F(r_n\abs{y_n})\,dx
		\ge
		\frac{|A|}{r_n^2}F(r_n\rho)
		\to\infty .
		\]
		The positive quadratic part of \(J_{\lambda_n,\G_{L_n}}(z_n)\) is \(O(r_n^2)\), whereas the negative quadratic part is nonpositive.  Hence
		\[
		J_{\lambda_n,\G_{L_n}}(z_n)\to-\infty,
		\]
		again contradicting \eqref{eq3.20}.  This proves \eqref{eq3.19}.
		
		We next localize the action from above.  Let \(\chi_R,\psi_R\) be real smooth cutoffs on the pendant, extended continuously to the whole graph, such that
		\[
		\chi_R^2+\psi_R^2=1,\qquad
		0\le\chi_R,\psi_R\le1,\qquad
		\chi_R=1\text{ on }[0,R],
		\]
		\[
		\chi_R=0\text{ outside }[0,2R],
		\qquad
		\abs{\chi_R'}+\abs{\psi_R'}\le C/R .
		\]
		We extend \(\chi_R\) by \(0\) outside the pendant and \(\psi_R\) by \(1\) outside the pendant.  For \(L_n>2R\), both are constant in a neighborhood of the attachment point, so these extensions are Lipschitz whether that point is a vertex or lies in the interior of an edge.
		
		By compactness of \(\{\omega_n\}\subset\Y_{\R_+}\),
		\[
		\sup_n\norm{\psi_R\omega_n}_{\Y_{\R_+}}=o_R(1).
		\]
		Using \eqref{eq3.16}, boundedness of \(t_n\), and boundedness of multiplication by \(\psi_R\),
		\begin{equation*}
			\norm{\psi_Rt_nw_n}_{\G_{L_n}}=o_R(1).
		\end{equation*}
		Moreover,
		\[
		\begin{split}
			P_{\G_{L_n}}^+(\psi_Rz_n)
			&=
			\psi_RP_{\G_{L_n}}^+z_n+
			[P_{\G_{L_n}}^+,\psi_R]z_n\\
			&=
			\psi_Rt_nw_n+
			[P_{\G_{L_n}}^+,\psi_R]z_n .
		\end{split}
		\]
		By \eqref{eq3.10} and \eqref{eq3.19},
		\[
		\norm{[P_{\G_{L_n}}^+,\psi_R]z_n}_{\G_{L_n}}
		\le
		\frac{C}{R}.
		\]
		Thus
		\begin{equation*}
			\norm{P_{\G_{L_n}}^+(\psi_Rz_n)}_{\G_{L_n}}
			=
			o_R(1).
		\end{equation*}
		The coercivity estimate gives
		\begin{equation*}
			Q_{\lambda_n,\G_{L_n}}(\psi_Rz_n)
			\le
			C_I\norm{P_{\G_{L_n}}^+(\psi_Rz_n)}_{\G_{L_n}}^2
			=
			o_R(1).
		\end{equation*}
		By \eqref{eq3.11},
		\[
		Q_{\lambda_n,\G_{L_n}}(z_n)
		=
		Q_{\lambda_n,\G_{L_n}}(\chi_Rz_n)
		+
		Q_{\lambda_n,\G_{L_n}}(\psi_Rz_n).
		\]
		Since \(0\le\chi_R\le1\) and \(F\) is nondecreasing,
		\[
		F(\abs{z_n})-F(\abs{\chi_Rz_n})\ge0.
		\]
		Consequently,
		\[
		\begin{split}
			J_{\lambda_n,\G_{L_n}}(z_n)
			-
			J_{\lambda_n,\G_{L_n}}(\chi_Rz_n)
			&=
			\frac12Q_{\lambda_n,\G_{L_n}}(\psi_Rz_n)\\
			&\quad-
			\int_{\G_{L_n}}
			\bigl(F(\abs{z_n})-F(\abs{\chi_Rz_n})\bigr)\,dx\\
			&\le o_R(1).
		\end{split}
		\]
		Hence
		\begin{equation}\label{eq3.22}
			J_{\lambda_n,\G_{L_n}}(z_n)
			\le
			J_{\lambda_n,\G_{L_n}}(\chi_Rz_n)+o_R(1).
		\end{equation}
		
		We prove the limsup inequality.  Fix \(R>0\).  Identify \(\chi_Rz_n\) with its half-line representative
		\[
		u_{R,n}:=R_{L_n}(\chi_Rz_n).
		\]
		The sequence \(u_{R,n}\) is bounded in \(\Y_{\R_+}\), supported in \([0,2R]\), and, after passing to a subsequence,
		\[
		u_{R,n}\rightharpoonup z_R\quad\text{in }\Y_{\R_+},
		\qquad
		u_{R,n}\to z_R\quad\text{in }L^q([0,2R])
		\]
		for every \(2\le q<\infty\).
		
		Since
		\[
		P_{\G_{L_n}}^+(\chi_Rz_n)
		=
		\chi_Rt_nw_n+[P_{\G_{L_n}}^+,\chi_R]z_n,
		\]
		\eqref{eq3.16} and the commutator estimate imply
		\[
		R_{L_n}P_{\G_{L_n}}^+(\chi_Rz_n)
		=
		\chi_Rt_n\omega_n+O(R^{-1})+o_n(1)
		\quad\text{in }\Y_{\R_+}.
		\]
		On the other hand, by the supported estimate \eqref{eq3.6},
		\[
		R_{L_n}P_{\G_{L_n}}^+(\chi_Rz_n)
		=
		P_+^+u_{R,n}+o_n(1)
		\quad\text{in }\Y_{\R_+}.
		\]
		Passing to a subsequence with \(t_n\to t\ge0\), we obtain
		\[
		\norm{P_+^+z_R-t\omega_*}_{\Y_{\R_+}}
		\le
		\norm{(1-\chi_R)t\omega_*}_{\Y_{\R_+}}+O(R^{-1})
		=
		o_R(1).
		\]
		Define
		\[
		\widehat z_R:=t\omega_*+P_+^-z_R
		\in\widehat\Y_{\R_+}(\omega_*).
		\]
		Then
		\begin{equation}\label{eq3.23}
			\norm{z_R-\widehat z_R}_{\Y_{\R_+}}=o_R(1).
		\end{equation}
		
		By \eqref{eq3.5},
		\[
		Q_{\lambda_n,\G_{L_n}}(\chi_Rz_n)
		=
		Q_{\lambda_n,+}(u_{R,n}).
		\]
		For fixed \(R\), set \(q_R=\chi_Rt\omega_*\).  The preceding two projection estimates give
		\[
		\limsup_{n\to\infty}
		\norm{P_+^+u_{R,n}-q_R}_{\Y_{\R_+}}
		\le \frac{C}{R},
		\qquad
		\norm{P_+^+z_R-q_R}_{\Y_{\R_+}}
		\le \frac{C}{R}.
		\]
		Since these vectors are uniformly bounded,
		\[
		\limsup_{n\to\infty}
		\norm{P_+^+u_{R,n}}_{\Y_{\R_+}}^2
		\le
		\norm{P_+^+z_R}_{\Y_{\R_+}}^2+o_R(1).
		\]
		The norm of the negative part is weakly lower semicontinuous (and hence its negative quadratic contribution is weakly upper semicontinuous), while the \(L^2\)-term converges strongly on \([0,2R]\).  Hence
		\[
		\limsup_{n\to\infty}
		Q_{\lambda_n,\G_{L_n}}(\chi_Rz_n)
		\le
		Q_{\lambda_*,+}(z_R)+o_R(1).
		\]
		The nonlinear term converges strongly on \([0,2R]\), because \(u_{R,n}\to z_R\) in \(L^q([0,2R])\) for every finite \(q\), and \(F\) has polynomial growth.  Therefore
		\[
		\limsup_{n\to\infty}
		J_{\lambda_n,\G_{L_n}}(\chi_Rz_n)
		\le
		J_{\lambda_*,\R_+}(z_R)+o_R(1).
		\]
		By \eqref{eq3.23} and continuity of \(J_{\lambda_*,\R_+}\) on bounded subsets of \(\Y_{\R_+}\),
		\[
		J_{\lambda_*,\R_+}(z_R)
		\le
		J_{\lambda_*,\R_+}(\widehat z_R)+o_R(1)
		\le
		\Gamma_{\lambda_*,\R_+}(\omega_*)+o_R(1).
		\]
		Together with \eqref{eq3.22}, this yields
		\[
		\limsup_{n\to\infty}
		\Gamma_{\lambda_n,\G_{L_n}}(w_n)
		\le
		\Gamma_{\lambda_*,\R_+}(\omega_*)+o_R(1).
		\]
		Letting \(R\to\infty\), we get
		\begin{equation}\label{eq3.24}
			\limsup_{n\to\infty}
			\Gamma_{\lambda_n,\G_{L_n}}(w_n)
			\le
			\Gamma_{\lambda_*,\R_+}(\omega_*).
		\end{equation}
		
		For the reverse inequality, observe again that
		\[
		v_{\lambda_n,L_n}
		\in\widehat\Y_{\G_{L_n}}(w_n).
		\]
		Therefore,
		\[
		\Gamma_{\lambda_n,\G_{L_n}}(w_n)
		\ge
		J_{\lambda_n,\G_{L_n}}(v_{\lambda_n,L_n}).
		\]
		By \eqref{eq3.12}, the locality of the action, and the fibre-maximizing assumption,
		\begin{equation}\label{eq3.25}
			\begin{split}
				\liminf_{n\to\infty}
				\Gamma_{\lambda_n,\G_{L_n}}(w_n)
				&\ge
				\lim_{n\to\infty}
				J_{\lambda_n,\R_+}(v_{\lambda_n})\\
				&=
				\lim_{n\to\infty}
				\Gamma_{\lambda_n,\R_+}(\omega_n)\\
				&=
				\Gamma_{\lambda_*,\R_+}(\omega_*).
			\end{split}
		\end{equation}
		
		Combining \eqref{eq3.24}, \eqref{eq3.25}, and \eqref{eq3.18}, we get \eqref{eq3.17}.  The sequential criterion proves the uniform convergence \eqref{eq3.13}.
	\end{proof}
	
	\begin{lemma}\label{Lem3.9}
		There exists a modulus \(\eta_I(L)\to0\) as \(L\to\infty\), depending only on \(I,m,c,f\), with the following property for all sufficiently large \(L\). Let \(\G\) be any generalized star graph carrying a distinguished pendant path \(\Pp\simeq[0,L]\), and let \(W_{\lambda,L}\) be the profile from \cref{Lem3.3} placed on \(\Pp\). Then
		\[
		P_\G^+W_{\lambda,L}\ne0
		\qquad(\lambda\in I),
		\]
		and
		\begin{equation*}
			\sup_{\lambda\in I}
			\left|
			\Gamma_{\lambda,\G}(P_\G^+W_{\lambda,L})-\frac12d_\lambda^\infty
			\right|
			\le\eta_I(L).
		\end{equation*}
		The modulus \(\eta_I\) is independent of the compact core, the number and lengths of the remaining edges, the vertex degrees, and the topology away from \(\Pp\).
	\end{lemma}
	
	\begin{proof}
		Apply \cref{Lem3.8} with
		\[
		v_{\lambda,L}=W_{\lambda,L},
		\qquad
		v_\lambda=U_{\lambda,+}.
		\]
		The convergence hypothesis follows from \eqref{eq3.2}, and the compactness of \(\{U_{\lambda,+}:\lambda\in I\}\) follows from \cref{Prop3.1}.  Since \(U_{\lambda,+}\in\N_{\lambda,\R_+}\), the additional fibre-maximizing hypothesis in \cref{Lem3.8} holds, and the fibre minimax formula gives
		\[
		\Gamma_{\lambda,\R_+}\big(P_+^+U_{\lambda,+}\big)
		=
		J_{\lambda,\R_+}(U_{\lambda,+})
		=
		\frac12d_\lambda^\infty
		\]
		by \cref{Lem3.2}.  This proves the claim.
	\end{proof}
	
	\begin{corollary}\label{Cor3.10}
		There exists \(L_I>0\) such that, if a generalized star graph contains a pendant path of length \(L\ge L_I\), then for every \(\lambda\in I\) there is \(w_{\lambda,L}\in\Y_\G^+\setminus\{0\}\) satisfying
		\[
		\Gamma_{\lambda,\G}(w_{\lambda,L})<d_\lambda^\infty .
		\]
		More precisely, one may choose \(L_I\) so that
		\begin{equation*}
			\Gamma_{\lambda,\G}(w_{\lambda,L})
			\le d_\lambda^\infty-\frac14d_I^*
			\qquad(\lambda\in I,\ L\ge L_I).
		\end{equation*}
	\end{corollary}
	
	\begin{proof}
		By \cref{Lem3.9}, choose \(L_I\ge1\) so large that the estimate in that lemma holds and
		\[
		\eta_I(L)\le\frac14d_I^*
		\qquad\text{for every }L\ge L_I .
		\]
		Set
		\[
		w_{\lambda,L}:=P_\G^+W_{\lambda,L}.
		\]
		For such \(L\), this vector is nonzero by \cref{Lem3.9}.  Since \(d_\lambda^\infty\ge d_I^*\),
		\[
		\Gamma_{\lambda,\G}(w_{\lambda,L})
		\le
		\frac12d_\lambda^\infty+\frac14d_I^*
		\le
		d_\lambda^\infty-\frac14d_I^* .
		\]
		This proves the claim.
	\end{proof}

	\section{Proofs of the main statements}
	
	\subsection{The strict gap theorem}
	
	\begin{proof}[Proof of \cref{Thm1.1}]
		Let \(L_I\) be the length given by \cref{Cor3.10}.  Assume that the essential reduction \(\G_{\ess}\) contains a pendant path \(\Pp_{\ess}\) of length \(L\ge L_I\).  By undoing the suppression of transmissive degree-two vertices along \(\Pp_{\ess}\), we obtain in the original graph \(\G\) a terminal metric interval \(\Pp\) of the same length \(L\).  It is a finite concatenation of edge segments separated only by transmissive degree-two vertices, and its endpoint corresponding to \(t=L\) may be either a vertex or an interior point of a bounded or external edge.  We identify \(\Pp\) with \([0,L]\), with \(0\) at the terminal endpoint, using the orientation convention in \eqref{eq2.1} whenever an edge has to be reoriented.
		
		Under this identification, an \(H^1\)-spinor of the concatenated coordinate satisfies the original Dirac-Kirchhoff conditions at every suppressed degree-two vertex: the first component is continuous, and the signed lower-component flux condition becomes continuity of the lower component in the concatenated coordinate.  All local profiles and cutoffs which are extended by zero in the pendant construction vanish in a neighborhood of \(L\); their zero extensions are therefore admissible also when \(L\) lies in the interior of an edge.  Thus the pendant construction in \cref{Lem3.3,Lem3.4} and the fibre estimate in \cref{Cor3.10} apply in every case allowed by \cref{Def2.3}.
		
		Hence, for every \(\lambda\in I\), there exists \(w_{\lambda,L}\in\Y_\G^+\setminus\{0\}\) such that
		\[
		\Gamma_{\lambda,\G}(w_{\lambda,L})
		\le
		d_\lambda^\infty-\frac14d_I^*,
		\qquad
		d_I^*=\inf_{\mu\in I}d_\mu^\infty>0 .
		\]
		By the fibre minimax formula in \cref{Prop2.15},
		\[
		d_\lambda(\G)
		\le
		\Gamma_{\lambda,\G}(w_{\lambda,L})
		\le
		d_\lambda^\infty-\frac14d_I^* .
		\]
		Therefore
		\[
		d_\lambda(\G)<d_\lambda^\infty
		\qquad(\lambda\in I),
		\]
		and the uniform estimate \eqref{eq1.2} follows with
		\[
		\rho_I=\frac14d_I^* .
		\]
		This proves \cref{Thm1.1}.
	\end{proof}
	
	\subsection{The rigidity results}
	
	\begin{proof}[Proof of \cref{Thm1.2}]
		Assume that \(\G_{\ess}\) is isometric to \(\R\).  Then \(\G\) is obtained from a metric realization of the real line by inserting only transmissive degree-two vertices.  Choose an isometric coordinate \(x\in\R\) on \(\G_{\ess}\).  Each edge of \(\G\) is then identified with an interval in this coordinate.
		
		If the original orientation of an edge agrees with the \(x\)-coordinate, we use the identity map on that edge.  If the original edge coordinate is opposite to the \(x\)-coordinate and has length \(\ell\), we use the unitary orientation reversal
		\[
		(u^1(x),u^2(x))
		\longmapsto
		(u^1(\ell-x),-u^2(\ell-x)).
		\]
		Equivalently, writing \(S=\operatorname{diag}(1,-1)\), this transformation has the form \(Tu(x)=S\,u(\ell-x)\).  Since \(S\sigma_1=-\sigma_1S\) and \(S\sigma_3=\sigma_3S\), a direct computation gives
		\[
		\left(-ic\sigma_1\frac{d}{dx}+mc^2\sigma_3\right)Tu
		=
		T\left(-ic\sigma_1\frac{d}{dx}+mc^2\sigma_3\right)u .
		\]
		Thus the Dirac expression is preserved in the concatenated coordinate.
		
		At a transmissive degree-two vertex, the Dirac-Kirchhoff condition says that the first component is continuous and that the signed lower-component flux is continuous.  After the above orientation convention, this is exactly the continuity of both components in the concatenated coordinate.  Therefore the local transformations on all edges combine to a unitary operator
		\[
		\mathcal U:L^2(\G,\C^2)\longrightarrow L^2(\R,\C^2)
		\]
		such that
		\[
		\mathcal U\dom(\D_\G)=\dom(\D_\R),
		\qquad
		\mathcal U\D_\G\mathcal U^{-1}=\D_\R .
		\]
		By functional calculus,
		\[
		\mathcal U\Y_\G^\pm=\Y_\R^\pm,
		\qquad
		\norm{\mathcal Uu}_{\R}=\norm{u}_{\G}.
		\]
		Moreover, \(\mathcal U\) preserves the \(L^2\)-norm and the pointwise modulus \(\abs{u}\).  Hence
		\[
		J_{\lambda,\R}(\mathcal Uu)=J_{\lambda,\G}(u),
		\qquad
		u\in\Y_\G .
		\]
		It follows that
		\[
		\mathcal U\N_{\lambda,\G}=\N_{\lambda,\R}.
		\]
		Taking infima over the two Nehari manifolds gives
		\[
		d_\lambda(\G)=d_{\lambda,\R}=d_\lambda^\infty .
		\]
		This proves \cref{Thm1.2}.
	\end{proof}
	
	\begin{proof}[Proof of \cref{Prop1.3}]
		Assume that \(\G_{\ess}\) is isometric to \(\R_+\).  Choose the coordinate \(x\in[0,\infty)\) on \(\G_{\ess}\) so that \(x=0\) is the terminal endpoint. As in the proof of \cref{Thm1.2}, concatenate all transmissive degree-two chains and use the orientation reversal \eqref{eq2.1} on every edge whose original coordinate is opposite to the half-line coordinate.
		
		This gives a unitary operator
		\[
		\mathcal U:L^2(\G,\C^2)\longrightarrow L^2(\R_+,\C^2)
		\]
		which intertwines the Dirac expressions.  At every suppressed degree-two vertex, the Dirac-Kirchhoff condition becomes continuity of the spinor in the concatenated coordinate.  At the terminal endpoint, the degree-one Dirac-Kirchhoff condition is precisely
		\[
		u_2(0)=0,
		\]
		which is the half-line endpoint condition used in the definition of \(d_\lambda^+\).  Hence
		\[
		\mathcal U\dom(\D_\G)=\dom(\D_+),
		\qquad
		\mathcal U\D_\G\mathcal U^{-1}=\D_+ .
		\]
		Functional calculus gives
		\[
		\mathcal U\Y_\G^\pm=\Y_{\R_+}^\pm .
		\]
		Since \(\mathcal U\) also preserves the quadratic form, the \(L^2\)-term, and the nonlinear integral, it identifies the action functionals:
		\[
		J_{\lambda,\R_+}(\mathcal Uu)=J_{\lambda,\G}(u),
		\qquad
		u\in\Y_\G .
		\]
		Consequently,
		\[
		\mathcal U\N_{\lambda,\G}=\N_{\lambda,\R_+}.
		\]
		Taking the infimum of the action over the two corresponding generalized Nehari manifolds yields
		\[
		d_\lambda(\G)=d_\lambda^+ .
		\]
		This proves \cref{Prop1.3}.
	\end{proof}
	
	\appendix
	
	\section{Full-line ground states}\label{AppA}
	
	The purpose of this appendix is to justify the one-dimensional result used in \cref{Prop3.1}.  We follow the autonomous variational method of Ding, Guo and Xu \cite{DingGuoXu2021}, verifying its application to \(H^{1/2}(\R,\C^2)\) under \textnormal{(F1)--(F4)}.  We then prove the normalization, compactness, decay and continuity properties needed in the main text.
	
	We write the full-line equation as
	\begin{equation}\label{eqA.1}
		(\D_\R+\lambda)U=f(|U|)U,
		\qquad
		U=(U_1,U_2)^T .
	\end{equation}
	Equivalently,
	\begin{equation}\label{eqA.2}
		\begin{split}
			-ic\,U_2'+(mc^2+\lambda)U_1&=f(|U|)U_1,\\
			-ic\,U_1'+(\lambda-mc^2)U_2&=f(|U|)U_2 .
		\end{split}
	\end{equation}
	
	\begin{theorem}\label{ThmA.1}
		Let \(\lambda\in(-mc^2,mc^2)\).  Under \textnormal{(F1)--(F4)}, the full-line functional \(J_{\lambda,\R}\) possesses a least-energy critical point \(U_\lambda\in\N_{\lambda,\R}\), namely
		\[
		J'_{\lambda,\R}(U_\lambda)=0,
		\qquad
		J_{\lambda,\R}(U_\lambda)=d_\lambda^\infty .
		\]
		Moreover,
		\[
		U_\lambda\in H^1(\R,\C^2)\cap C^1(\R,\C^2),
		\]
		\(U_\lambda\) solves \eqref{eqA.1} classically, and
		\[
		U_\lambda(x)\to0
		\qquad\text{as } |x|\to\infty .
		\]
	\end{theorem}
	
	\begin{proof}
		We adapt the autonomous variational argument of Ding, Guo and Xu \cite[Section 3.1]{DingGuoXu2021}, verifying the required properties under \textnormal{(F1)--(F4)}.  Fix \(\lambda\in(-mc^2,mc^2)\), and write
		\[
		E=\Y_\R=H^{1/2}(\R,\C^2)=E^+\oplus E^-,
		\qquad J=J_{\lambda,\R},\qquad d=d_\lambda^\infty.
		\]
		Indeed, the Fourier symbol of \(|\D_\R|\) is \((c^2\xi^2+m^2c^4)^{1/2}I_2\).  On \(E\) we use the equivalent norm
		\[
		\norm{u}_\lambda^2
		=Q_{\lambda,\R}(u^+)-Q_{\lambda,\R}(u^-).
		\]
		Then
		\[
		J(u)=\frac12\bigl(\norm{u^+}_\lambda^2-
		\norm{u^-}_\lambda^2\bigr)-\Psi(u),
		\qquad \Psi(u)=\int_\R F(|u|)\,\dd x.
		\]
		The embeddings \(E\hookrightarrow L^q(\R,\C^2)\) hold for every \(2\le q<\infty\), and restriction to a bounded interval is compact into each such \(L^q\) space.  Consequently \(\Psi\) is of class \(C^1\). It is nonnegative and weakly lower semicontinuous by convexity.  Moreover, if \(u_n\rightharpoonup u\) in \(E\), local strong convergence and \eqref{eq2.2} imply \(\Psi'(u_n)[\varphi]\to\Psi'(u)[\varphi]\) for every \(\varphi\in C_c^\infty(\R,\C^2)\).  The derivatives are bounded in \(E^*\), so density gives the same convergence for every \(\varphi\in E\).
		
		By \eqref{eq2.2}, there exist \(r,b>0\) such that
		\[
		J(w)\ge b\qquad
		(w\in E^+,\ \norm{w}_\lambda=r).
		\]
		We use the abstract Cerami linking theorem in \cite{Tang2015}, regarding \(E\) as a real Hilbert space.  Its assumptions on the nonlinear part are exactly the properties of \(\Psi\) verified above. The proof of \cref{Prop2.15} shows that \(J\) tends to \(-\infty\) on each fibre \(\widehat E(w)=\{tw+\eta:t\ge0,\ \eta\in E^-\}\) as the norm tends to infinity.  The fibre argument uses \textnormal{(F4)} with nondecreasing \(f\), and uses the lower bound \eqref{eq2.4} only at infinity.  Thus \(0<b\le d<\infty\).  Choose \(w_n\in E^+\) with \(\norm{w_n}_\lambda=1\) and
		\[
		\max_{\widehat E(w_n)}J\le d+\frac1n.
		\]
		For sufficiently large \(R_n>r\), the truncated cone
		\[
		Q_n=\{tw_n+\eta:t\ge0,\ \eta\in E^-,\,
		\norm{tw_n+\eta}_\lambda\le R_n\}
		\]
		satisfies \(\sup_{\partial Q_n}J\le0<b\), where the boundary is relative to \(E^-\oplus\R w_n\).  Indeed, \(J\le0\) on the base \(t=0\), and fibre anti-coercivity gives the same bound on the outer boundary.  The linking theorem applied to \(J\) therefore gives a Cerami sequence at some level \(\ell_n\in[b,d+1/n]\).  Selecting one sufficiently late term from each of these sequences, and passing to a subsequence, yields
		\[
		J(u_n)\to\ell\in[b,d],\qquad
		(1+\norm{u_n}_\lambda)\norm{J'(u_n)}_{E^*}\to0.
		\]
		No uniqueness of the maximizer on a fibre is needed.
		
		We prove that \((u_n)\) is bounded.  By \eqref{eq2.3},
		\[
		\begin{split}
			J(u_n)-\frac12J'(u_n)[u_n]
			&=\int_\R H(|u_n|)\,\dd x\\
			&\ge\left(\frac12-\frac1\theta\right)
			\int_\R f(|u_n|)|u_n|^2\,\dd x.
		\end{split}
		\]
		Hence \(A_n:=\int_\R f(|u_n|)|u_n|^2\,\dd x\) is bounded. Put \(p'=p/(p-1)\).  Given \(\eps>0\), choose \(\delta>0\) so that \(f(t)\le\eps\) for \(0\le t\le\delta\).  By \textnormal{(F2)},
		\[
		(f(t)t)^{p'}\le C_\delta f(t)t^2
		\qquad(t\ge\delta).
		\]
		Testing \(J'(u_n)\) with \(u_n^+-u_n^-\), and estimating separately on \(\{|u_n|<\delta\}\) and \(\{|u_n|\ge\delta\}\), gives
		\[
		\norm{u_n}_\lambda^2
		\le C\eps\norm{u_n}_\lambda^2
		+C_\delta A_n^{1/p'}\norm{u_n}_\lambda+o(1).
		\]
		Here we used H\"older's inequality and \(E\hookrightarrow L^p\). Choosing \(\eps\) sufficiently small proves boundedness.
		
		We next exclude vanishing.  Interpolation and the \(H^{1/2}\) embedding into \(L^4\) on unit intervals give
		\[
		\norm{z}_{L^3(\R)}^3
		\le C\left(\sup_{j\in\mathbb Z}
		\norm{z}_{L^2(j,j+1)}\right)\norm{z}_{H^{1/2}(\R)}^2.
		\]
		If the supremum on the right tends to zero for \(z=u_n\), then \(u_n\to0\) in \(L^3\), and interpolation with the bounded \(L^q\) norms gives convergence to zero in every \(L^q\), \(2<q<\infty\). Since \(f(t)t\le\eps t+C_\eps t^{p-1}\), testing once more with \(u_n^+-u_n^-\) yields
		\[
		\norm{u_n}_\lambda^2
		\le C\eps\norm{u_n}_\lambda^2
		+C_\eps\norm{u_n}_{L^p}^{p-1}\norm{u_n}_\lambda+o(1).
		\]
		It follows that \(u_n\to0\) in \(E\), contradicting \(\ell\ge b>0\). Thus, after passing to a subsequence, there are \(y_n\in\mathbb Z\) and \(a>0\) such that
		\[
		\int_{y_n}^{y_n+1}|u_n|^2\,\dd x\ge a.
		\]
		Translations commute with \(\D_\R\) and its spectral projections and preserve \(J\).  The translated sequence \(v_n=u_n(\,\cdot+y_n)\) therefore has the same Cerami properties. Passing to a further subsequence, we have
		\[
		v_n\rightharpoonup U\ \text{in }E,\qquad
		v_n\to U\ \text{in }L^q_{\rm loc}(\R,\C^2)
		\quad(2\le q<\infty),
		\]
		and almost everywhere.  Local \(L^2\) convergence gives \(U\ne0\). The weak continuity of \(\Psi'\) proved above implies \(J'(U)=0\).
		
		The positive component of \(U\) is nonzero.  Otherwise, testing the equation with \(U\) would give
		\[
		-\norm{U}_\lambda^2
		=\int_\R f(|U|)|U|^2\,\dd x\ge0,
		\]
		which is impossible.  Hence \(U\in\N_{\lambda,\R}\). Since \(H\ge0\), Fatou's lemma now gives
		\[
		\begin{split}
			d\le J(U)
			&=\int_\R H(|U|)\,\dd x\\
			&\le\liminf_{n\to\infty}\int_\R H(|v_n|)\,\dd x
			=\ell\le d.
		\end{split}
		\]
		Thus \(J(U)=d\), and every nonzero critical point has energy at least \(d\), since it belongs to \(\N_{\lambda,\R}\).
		
		Finally, \(U\in L^2\cap L^{2(p-1)}\) and \textnormal{(F2)} imply \(f(|U|)U\in L^2\).  Equation \eqref{eqA.1}, initially in distributions, gives \(U'\in L^2\), hence \(U\in H^1(\R,\C^2)\). The equation then has a continuous right-hand side, so \(U\in C^1\). The one-dimensional \(H^1\) embedding also gives \(U(x)\to0\) as \(|x|\to\infty\).  Taking \(U_\lambda=U\) completes the proof.
	\end{proof}
	
	\begin{lemma}\label{LemA.2}
		Let \(U\in H^1(\R,\C^2)\cap C^1(\R,\C^2)\) be a nontrivial solution of \eqref{eqA.1} such that \(U(x)\to0\) as \(|x|\to\infty\). Then, after multiplication by a constant complex phase, \(U\) can be written as
		\[
		U(x)=
		\begin{pmatrix}
			u(x)\\ i v(x)
		\end{pmatrix},
		\qquad
		u,v:\R\to\R .
		\]
		The functions \(u,v\) solve the real planar system
		\begin{equation}\label{eqA.3}
			\begin{split}
				c\,v'&=\bigl(f((u^2+v^2)^{1/2})-mc^2-\lambda\bigr)u,\\
				c\,u'&=\bigl(\lambda-mc^2-f((u^2+v^2)^{1/2})\bigr)v .
			\end{split}
		\end{equation}
	\end{lemma}
	
	\begin{proof}
		Define the Dirac current
		\[
		j(x):=2\Ree\bigl(U_1(x)\overline{U_2(x)}\bigr).
		\]
		From \eqref{eqA.2},
		\[
		U_2'
		=
		\frac{i}{c}\bigl(f(|U|)-mc^2-\lambda\bigr)U_1,
		\qquad
		U_1'
		=
		\frac{i}{c}\bigl(f(|U|)+mc^2-\lambda\bigr)U_2 .
		\]
		Hence
		\[
		\begin{split}
			j'(x)
			&=
			2\Ree\bigl(U_1'\overline{U_2}
			+U_1\overline{U_2'}\bigr)\\
			&=
			2\Ree\left[
			\frac{i}{c}\bigl(f(|U|)+mc^2-\lambda\bigr)|U_2|^2
			-
			\frac{i}{c}\bigl(f(|U|)-mc^2-\lambda\bigr)|U_1|^2
			\right]
			=
			0 .
		\end{split}
		\]
		Since \(U(x)\to0\) at both infinities, \(j\equiv0\).  Therefore
		\[
		\Ree\bigl(U_1\overline{U_2}\bigr)=0
		\qquad(x\in\R).
		\]
		
		We also note that \(U\) never vanishes.  Indeed, if \(U(x_0)=0\) for some \(x_0\), uniqueness for the Cauchy problem associated with the locally Lipschitz system \eqref{eqA.2} gives \(U\equiv0\), contrary to the assumption.
		
		Thus, locally on \(\R\), the relation \(\Ree(U_1\overline{U_2})=0\) allows us to write
		\[
		U=e^{i\vartheta(x)}(u,iv)^T,
		\qquad
		u,v\in\R .
		\]
		Substituting this expression into \eqref{eqA.2} and comparing real and imaginary parts gives
		\[
		\vartheta'(x)u(x)=0,
		\qquad
		\vartheta'(x)v(x)=0 .
		\]
		Since \((u(x),v(x))\ne(0,0)\) for every \(x\), we obtain \(\vartheta'\equiv0\).  Hence the phase is globally constant.  Multiplying by the inverse constant phase gives the desired representation.  Substituting \(U=(u,iv)^T\) into \eqref{eqA.2} gives \eqref{eqA.3}.
	\end{proof}
	
	\begin{lemma}\label{LemA.3}
		For every real solution \((u,v)\) of \eqref{eqA.3}, the quantity
		\[
		\mathcal H_\lambda(u,v)
		=
		\frac{mc^2+\lambda}{2}u^2
		-
		\frac{mc^2-\lambda}{2}v^2
		-
		F\bigl((u^2+v^2)^{1/2}\bigr)
		\]
		is constant along the orbit.  If the solution is homoclinic to \(0\), then
		\[
		\mathcal H_\lambda(u(x),v(x))=0
		\qquad(x\in\R).
		\]
	\end{lemma}
	
	\begin{proof}
		Set \(r=(u^2+v^2)^{1/2}\).  Since \(F'(r)=f(r)r\),
		\[
		\partial_u\mathcal H_\lambda
		=
		(mc^2+\lambda)u-f(r)u
		=
		-c\,v',
		\]
		and
		\[
		\partial_v\mathcal H_\lambda
		=
		-(mc^2-\lambda)v-f(r)v
		=
		c\,u',
		\]
		where \eqref{eqA.3} was used.  Therefore
		\[
		\frac{d}{dx}\mathcal H_\lambda(u(x),v(x))
		=
		\partial_u\mathcal H_\lambda\,u'
		+
		\partial_v\mathcal H_\lambda\,v'
		=
		-cv'u'+cu'v'=0 .
		\]
		For a homoclinic solution, \((u(x),v(x))\to(0,0)\) as \(|x|\to\infty\), and \(\mathcal H_\lambda(0,0)=0\).  Hence the constant is zero.
	\end{proof}
	
	\begin{lemma}\label{LemA.4}
		Let \(U_\lambda\) be a nontrivial full-line ground state.  Then, after a translation and multiplication by a constant complex phase, it may be chosen so that
		\[
		U_{\lambda,1}\ \text{is even},
		\qquad
		U_{\lambda,2}\ \text{is odd}.
		\]
		Moreover, in this normalization,
		\[
		U_{\lambda,2}(0)=0 .
		\]
	\end{lemma}
	
	\begin{proof}
		By \cref{LemA.2}, after a constant phase rotation,
		\[
		U_\lambda=(u,iv)^T
		\]
		with \(u,v\) real and satisfying \eqref{eqA.3}.  The function \(u\) is not identically zero.  Otherwise the second equation in \eqref{eqA.3} gives \(v\equiv0\), contradicting the nontriviality of \(U_\lambda\).
		
		Since \(u(x)\to0\) as \(|x|\to\infty\), the continuous function \(|u|\) attains a positive maximum at some point \(x_0\).  Multiplying by \(-1\), if necessary, we may assume
		\[
		u(x_0)=\max_{x\in\R}|u(x)|>0 .
		\]
		Then \(u'(x_0)=0\).  The second equation in \eqref{eqA.3} gives
		\[
		\bigl(\lambda-mc^2-f((u(x_0)^2+v(x_0)^2)^{1/2})\bigr)v(x_0)=0 .
		\]
		The coefficient is strictly negative, because \(\lambda<mc^2\) and \(f\ge0\). Thus \(v(x_0)=0\).
		
		The real system \eqref{eqA.3} is reversible under
		\[
		x\mapsto 2x_0-x,
		\qquad
		(u,v)\mapsto(u,-v).
		\]
		Indeed, if \((u,v)\) is a solution, then
		\[
		\widetilde u(x):=u(2x_0-x),
		\qquad
		\widetilde v(x):=-v(2x_0-x)
		\]
		is also a solution.  At \(x=x_0\),
		\[
		\widetilde u(x_0)=u(x_0),
		\qquad
		\widetilde v(x_0)=v(x_0)=0 .
		\]
		By uniqueness for the Cauchy problem of \eqref{eqA.3},
		\[
		\widetilde u=u,
		\qquad
		\widetilde v=v .
		\]
		Therefore
		\[
		u(x_0+x)=u(x_0-x),
		\qquad
		v(x_0+x)=-v(x_0-x).
		\]
		Translating \(x_0\) to the origin gives the stated parity.  Since the second component is \(iv\), it is odd and vanishes at \(0\).
	\end{proof}
	
	\begin{lemma}\label{LemA.5}
		Let \(I\) be a compact interval with \(I\Subset(-mc^2,mc^2)\), and choose for each \(\lambda\in I\) the centered representative from \cref{LemA.4}, with real functions \(u_\lambda,v_\lambda\),
		\[
		U_\lambda=(u_\lambda,iv_\lambda)^T,
		\qquad
		u_\lambda(0)=\max_{x\in\R}|u_\lambda(x)|>0,
		\qquad
		v_\lambda(0)=0 .
		\]
		Then there exist constants \(0<b_I<B_I<\infty\) such that
		\[
		b_I\le u_\lambda(0)\le B_I
		\qquad(\lambda\in I).
		\]
		Moreover,
		\[
		\sup_{\lambda\in I}\norm{U_\lambda}_{L^\infty(\R)}<\infty .
		\]
	\end{lemma}
	
	\begin{proof}
		Set
		\[
		a_\lambda:=u_\lambda(0)>0.
		\]
		At \(x=0\), the Hamiltonian identity in \cref{LemA.3} gives
		\[
		F(a_\lambda)
		=
		\frac{mc^2+\lambda}{2}a_\lambda^2 .
		\]
		Since \(I\Subset(-mc^2,mc^2)\),
		\[
		0<
		A_-(I):=\frac12\inf_{\lambda\in I}(mc^2+\lambda)
		\le
		\frac12(mc^2+\lambda)
		\le
		A_+(I):=\frac12\sup_{\lambda\in I}(mc^2+\lambda)
		<\infty .
		\]
		
		Because \(f(0)=0\) and \(f\) is continuous,
		\[
		F(t)=\int_0^t f(s)s\,\dd s=o(t^2)
		\qquad(t\downarrow0).
		\]
		Therefore the identity above excludes \(a_\lambda\to0\) uniformly for \(\lambda\in I\).
		
		On the other hand, by \textnormal{(F3)},
		\[
		\frac{d}{dt}\left(\frac{F(t)}{t^2}\right)
		=
		\frac{f(t)t^2-2F(t)}{t^3}
		\ge
		\frac{(\theta-2)F(t)}{t^3}>0
		\qquad(t>0).
		\]
		Thus \(t\mapsto F(t)t^{-2}\) is strictly increasing on \((0,\infty)\).  Also, by \eqref{eq2.4},
		\[
		F(t)t^{-2}\to\infty
		\qquad(t\to\infty).
		\]
		Hence the identity
		\[
		\frac{F(a_\lambda)}{a_\lambda^2}
		=
		\frac{mc^2+\lambda}{2}
		\]
		also gives a uniform upper bound for \(a_\lambda\).
		
		For arbitrary \(x\), the zero Hamiltonian identity gives
		\[
		\frac{mc^2+\lambda}{2}u_\lambda(x)^2
		-
		\frac{mc^2-\lambda}{2}v_\lambda(x)^2
		=
		F(|U_\lambda(x)|)\ge0 .
		\]
		Consequently
		\[
		|v_\lambda(x)|^2
		\le
		\frac{mc^2+\lambda}{mc^2-\lambda}u_\lambda(x)^2
		\le C_Iu_\lambda(x)^2 .
		\]
		Therefore
		\[
		|U_\lambda(x)|\le C_I|u_\lambda(x)|
		\le C_Iu_\lambda(0).
		\]
		The uniform upper bound for \(u_\lambda(0)\) gives the asserted \(L^\infty\)-bound.
	\end{proof}
	
	\begin{lemma}\label{LemA.6}
		Let \(I\) be a compact interval with \(I\Subset(-mc^2,mc^2)\), and let \(U_\lambda\) be the centered even-odd representatives above.  Then, for every \(\delta>0\), there exists \(R_\delta>0\) such that
		\[
		\sup_{\lambda\in I}\sup_{|x|\ge R_\delta}|U_\lambda(x)|\le\delta .
		\]
	\end{lemma}
	
	\begin{proof}
		By parity it is enough to consider \(x\ge0\).  Put
		\[
		a_\lambda=u_\lambda(0).
		\]
		By \cref{LemA.5}, \(a_\lambda\) remains in a compact subinterval of \((0,\infty)\).
		
		We first describe the zero Hamiltonian curve.  Write
		\[
		A_\lambda=\frac{mc^2+\lambda}{2},
		\qquad
		B_\lambda=\frac{mc^2-\lambda}{2}.
		\]
		The centered solution satisfies
		\[
		F(a_\lambda)=A_\lambda a_\lambda^2 .
		\]
		Since \(t\mapsto F(t)t^{-2}\) is strictly increasing and ranges from \(0\) to \(\infty\), this equation determines \(a_\lambda\) uniquely, and \(\lambda\mapsto a_\lambda\) is continuous.
		
		For each \(u\in(0,a_\lambda)\), the equation
		\begin{equation}\label{eqA.4}
			A_\lambda u^2
			-
			B_\lambda v^2
			-
			F\bigl((u^2+v^2)^{1/2}\bigr)=0
		\end{equation}
		has a unique solution \(v=V_\lambda(u)>0\).  Indeed, the left-hand side of \eqref{eqA.4}, viewed as a function of \(v^2\), is strictly decreasing.  At \(v=0\) it equals
		\[
		A_\lambda u^2-F(u)>0,
		\]
		because \(u<a_\lambda\) and \(F(t)t^{-2}\) is strictly increasing.  As \(|v|\to\infty\), it tends to \(-\infty\).  Hence the positive branch \(V_\lambda(u)\) is well-defined.  By the implicit function theorem it is continuous in \((\lambda,u)\) away from the endpoint \(u=a_\lambda\).
		
		At \(x=0\),
		\[
		c\,v_\lambda'(0)
		=
		\bigl(f(a_\lambda)-mc^2-\lambda\bigr)a_\lambda .
		\]
		Using \textnormal{(F3)} and
		\[
		F(a_\lambda)=\frac{mc^2+\lambda}{2}a_\lambda^2,
		\]
		we get
		\[
		f(a_\lambda)
		\ge
		\frac{\theta F(a_\lambda)}{a_\lambda^2}
		=
		\frac{\theta}{2}(mc^2+\lambda).
		\]
		Thus
		\[
		f(a_\lambda)-mc^2-\lambda
		\ge
		\left(\frac{\theta}{2}-1\right)(mc^2+\lambda),
		\]
		which is bounded below by a positive constant for \(\lambda\in I\).  Hence \(v_\lambda'(0)>0\), and for small positive \(x\) the orbit lies on the positive branch \(v=V_\lambda(u)\).  By uniqueness of the ODE and by the zero Hamiltonian identity, the orbit cannot leave this branch before reaching the origin.  Therefore, for \(x>0\),
		\[
		v_\lambda(x)>0,
		\qquad
		0<u_\lambda(x)<a_\lambda .
		\]
		The second equation in \eqref{eqA.3} then gives
		\[
		c\,u_\lambda'
		=
		\bigl(\lambda-mc^2-f(|U_\lambda|)\bigr)v_\lambda<0 .
		\]
		Thus \(u_\lambda\) is strictly decreasing on the right half-line.
		
		Fix \(\delta>0\).  By \cref{LemA.5}, there is \(C_I>0\) such that
		\[
		|U_\lambda(x)|\le C_Iu_\lambda(x)
		\qquad(x\ge0,\ \lambda\in I).
		\]
		Choose
		\[
		0<\delta_0<\frac12\inf_{\lambda\in I}a_\lambda
		\]
		so small that \(C_I\delta_0\le\delta\).  It remains to prove that the time needed for \(u_\lambda\) to decrease from \(a_\lambda\) to \(\delta_0\) is uniformly bounded.
		
		Along the positive branch, write
		\[
		r_\lambda(u)=\bigl(u^2+V_\lambda(u)^2\bigr)^{1/2}.
		\]
		Since
		\[
		c\,u_\lambda'
		=
		-\bigl(mc^2-\lambda+f(r_\lambda(u_\lambda))\bigr)
		V_\lambda(u_\lambda),
		\]
		the time needed to move from \(a_\lambda\) to \(\delta_0\) is
		\[
		T_{\delta_0}(\lambda)
		=
		\int_{\delta_0}^{a_\lambda}
		\frac{c\,\dd u}
		{\bigl(mc^2-\lambda+f(r_\lambda(u))\bigr)V_\lambda(u)} .
		\]
		The denominator is positive.  On compact subintervals \(\delta_0\le u\le a_\lambda-\varepsilon\), the integrand depends continuously on \((\lambda,u)\) and is uniformly bounded.
		
		It remains to consider the endpoint \(u=a_\lambda\).  Let
		\[
		G_\lambda(u,w)
		=
		A_\lambda u^2
		-
		B_\lambda w
		-
		F\bigl((u^2+w)^{1/2}\bigr),
		\qquad w=v^2.
		\]
		Then \(G_\lambda(a_\lambda,0)=0\), and
		\[
		\partial_uG_\lambda(a_\lambda,0)
		=
		a_\lambda\bigl(2A_\lambda-f(a_\lambda)\bigr)<0,
		\]
		while
		\[
		\partial_wG_\lambda(a_\lambda,0)
		=
		-B_\lambda-\frac12f(a_\lambda)<0 .
		\]
		The inequalities are uniform for \(\lambda\in I\), because \(a_\lambda\) stays in a compact subinterval of \((0,\infty)\) and \(f(a_\lambda)-2A_\lambda\) is uniformly positive.  Thus
		\[
		V_\lambda(u)^2
		=
		w_\lambda(u)
		=
		c_\lambda(a_\lambda-u)+O((a_\lambda-u)^2)
		\qquad(u\uparrow a_\lambda),
		\]
		with \(c_\lambda\) bounded above and below by positive constants uniformly for \(\lambda\in I\).  Hence the endpoint singularity of \(1/V_\lambda(u)\) is of order \((a_\lambda-u)^{-1/2}\), uniformly in \(\lambda\), and is integrable.  Therefore
		\[
		\sup_{\lambda\in I}T_{\delta_0}(\lambda)<\infty .
		\]
		Let \(R_\delta:=\sup_{\lambda\in I}T_{\delta_0}(\lambda)\).  If \(x\ge R_\delta\), then \(u_\lambda(x)\le\delta_0\), hence
		\[
		|U_\lambda(x)|\le C_I\delta_0\le\delta .
		\]
		The left half-line follows by parity.
	\end{proof}
	
	\begin{lemma}\label{LemA.7}
		Let \(I\) be a compact interval with \(I\Subset(-mc^2,mc^2)\), and choose the centered even-odd ground states above.  Then there exist \(C_I,\alpha_I>0\) such that
		\[
		|U_\lambda(x)|+|U_\lambda'(x)|
		\le C_Ie^{-\alpha_I|x|}
		\qquad(x\in\R,\ \lambda\in I).
		\]
	\end{lemma}
	
	\begin{proof}
		Let
		\[
		\mu_I:=\dist\bigl(I,\{-mc^2,mc^2\}\bigr)>0 .
		\]
		Choose \(\delta>0\) so small that
		\[
		f(s)\le \frac{\mu_I}{4}
		\qquad(0\le s\le\delta).
		\]
		By \cref{LemA.6}, there exists \(R_I>0\) such that
		\[
		|U_\lambda(x)|\le\delta
		\qquad(|x|\ge R_I,\ \lambda\in I).
		\]
		Set \(\omega=mc^2\), \(r_\lambda=|U_\lambda|\), and \(g(r)=2F(r)/r^2\) for \(r>0\). On the right half-line, \(u_\lambda,v_\lambda>0\) by the proof of \cref{LemA.6}. The Hamiltonian identity and \eqref{eqA.3} give
		\[
		\frac{u_\lambda^2}{r_\lambda^2}
		=\frac{\omega-\lambda+g(r_\lambda)}{2\omega},
		\qquad
		\frac{v_\lambda^2}{r_\lambda^2}
		=\frac{\omega+\lambda-g(r_\lambda)}{2\omega},
		\qquad
		r_\lambda r_\lambda'=-\frac{2\omega}{c}u_\lambda v_\lambda.
		\]
		By \textnormal{(F4)}, \(0\le g(r)\le f(r)\). Thus, for \(x\ge R_I\),
		\[
		\frac{r_\lambda'}{r_\lambda}
		=-\frac1c\sqrt{\bigl(\omega-\lambda+g(r_\lambda)\bigr)
			\bigl(\omega+\lambda-g(r_\lambda)\bigr)}
		\le-\frac{\sqrt3\,\mu_I}{2c}.
		\]
		Integrating this inequality, using parity and the uniform amplitude bound from \cref{LemA.5}, gives
		\[
		|U_\lambda(x)|\le C_Ie^{-\alpha_I|x|}
		\qquad(x\in\R,\ \lambda\in I),
		\qquad \alpha_I=\frac{\sqrt3\,\mu_I}{2c}.
		\]
		
		Finally, from \eqref{eqA.2},
		\[
		|U_\lambda'(x)|
		\le
		C_I\bigl(1+f(|U_\lambda(x)|)\bigr)|U_\lambda(x)|.
		\]
		The family \(\{U_\lambda\}\) is uniformly bounded in \(L^\infty\) by \cref{LemA.5}; hence \(f(|U_\lambda|)\) is uniformly bounded. The exponential estimate for \(U_\lambda\) therefore gives the same exponential estimate for \(U_\lambda'\), after increasing \(C_I\).
	\end{proof}
	
	\begin{lemma}\label{LemA.8}
		Let \(I\) be a compact interval with \(I\Subset(-mc^2,mc^2)\).  The family \(\{U_\lambda:\lambda\in I\}\), with the centered even-odd normalization, is compact in \(H^1(\R,\C^2)\).
	\end{lemma}
	
	\begin{proof}
		Let \(\lambda_n\in I\).  Passing to a subsequence, assume
		\[
		\lambda_n\to\lambda_*\in I .
		\]
		By \cref{LemA.5}, the sequence \(U_{\lambda_n}\) is uniformly bounded in \(L^\infty(\R)\).  The ODE system \eqref{eqA.2} then gives a uniform \(C^1_{\rm loc}\) bound.  Since the right-hand side of the ODE is locally Lipschitz uniformly on bounded sets of \(U\) and \(\lambda\), the derivatives are locally equicontinuous.  By Arzelà-Ascoli, after passing to a further subsequence,
		\[
		U_{\lambda_n}\to U_*
		\qquad\text{in }C^1_{\rm loc}(\R,\C^2).
		\]
		The limit solves \eqref{eqA.1} with parameter \(\lambda_*\), has the same even-odd normalization, and is nontrivial because
		\[
		u_{\lambda_n}(0)\ge b_I>0 .
		\]
		
		By \cref{LemA.7},
		\[
		|U_{\lambda_n}(x)|+|U_{\lambda_n}'(x)|
		\le C_Ie^{-\alpha_I|x|}
		\qquad(n\in\mathbb N,\ x\in\R).
		\]
		Hence
		\[
		\sup_n
		\int_{|x|\ge R}
		\left(|U_{\lambda_n}(x)|^2+|U_{\lambda_n}'(x)|^2\right)\,\dd x
		\to0
		\qquad(R\to\infty).
		\]
		The local \(C^1\)-convergence, together with this uniform tail estimate, implies
		\[
		U_{\lambda_n}\to U_*
		\qquad\text{strongly in }H^1(\R,\C^2).
		\]
		By the continuity of \(\lambda\mapsto a_\lambda\) established in \cref{LemA.6},
		\[
		U_*(0)=(a_{\lambda_*},0)^T=U_{\lambda_*}(0).
		\]
		Both functions solve the same first-order system with parameter \(\lambda_*\), so uniqueness for the Cauchy problem gives \(U_*=U_{\lambda_*}\).  Thus every sequence in the family has a subsequence converging to an element of the family, which proves compactness in \(H^1(\R,\C^2)\).
	\end{proof}
	
	\begin{lemma}\label{LemA.9}
		Let \(I\) be a compact interval with \(I\Subset(-mc^2,mc^2)\).  The map
		\[
		\lambda\mapsto d_\lambda^\infty
		\]
		is continuous on \(I\), and
		\[
		\inf_{\lambda\in I}d_\lambda^\infty>0 .
		\]
	\end{lemma}
	
	\begin{proof}
		Let \(\lambda_n\to\lambda\) in \(I\).  By \cref{LemA.8}, after passing to a subsequence,
		\[
		U_{\lambda_n}\to U
		\qquad\text{strongly in }H^1(\R,\C^2).
		\]
		Passing to the equation gives
		\[
		(\D_\R+\lambda)U=f(|U|)U .
		\]
		Moreover \(U\not\equiv0\), because the centered amplitudes satisfy
		\[
		u_{\lambda_n}(0)\ge b_I>0 .
		\]
		Thus \(J'_{\lambda,\R}(U)=0\) and \(U\ne0\).
		
		A nonzero critical point cannot belong to \(\Y_{\R}^-\).  Indeed, if \(U\in\Y_{\R}^-\), then testing the critical equation with \(U\) gives
		\[
		Q_{\lambda,\R}(U)
		=
		\int_\R f(|U|)|U|^2\,\dd x\ge0.
		\]
		But the quadratic form is strictly negative on \(\Y_{\R}^-\setminus\{0\}\) because \(\lambda\in(-mc^2,mc^2)\).  This is a contradiction.  Hence \(U\notin\Y_{\R}^-\), and since \(U\) is a nonzero critical point,
		\[
		U\in\N_{\lambda,\R}.
		\]
		Therefore
		\[
		d_\lambda^\infty
		\le
		J_{\lambda,\R}(U)
		=
		\lim_{n\to\infty}J_{\lambda_n,\R}(U_{\lambda_n})
		=
		\lim_{n\to\infty}d_{\lambda_n}^\infty
		\]
		along the chosen subsequence.  This gives lower semicontinuity.
		
		For the reverse inequality, fix the ground state \(U_\lambda\) at the limiting parameter and set
		\[
		w=P_\R^+U_\lambda\in\Y_\R^+\setminus\{0\}.
		\]
		By the fibre minimax formula,
		\[
		d_{\lambda_n}^\infty
		\le
		\Gamma_{\lambda_n,\R}(w).
		\]
		The fibre anti-coercivity argument in \cref{Prop2.15}, applied to the fixed positive direction \(w\), gives a uniform bound on all maximizers of \(J_{\lambda_n,\R}\) over \(\widehat\Y_\R(w)\), for \(n\) large.  Since \(J_{\lambda_n,\R}\to J_{\lambda,\R}\) on bounded subsets of \(\Y_\R\), the standard fibre-stability argument gives
		\[
		\Gamma_{\lambda_n,\R}(w)
		\to
		\Gamma_{\lambda,\R}(w).
		\]
		Indeed, the liminf inequality is obtained by testing the \(n\)-th fibre with the fixed limiting maximizer, while the limsup inequality follows by taking bounded maximizers on the \(n\)-th fibres, passing weakly to the limiting fibre, and using weak upper semicontinuity exactly as in \cref{Prop2.15}.  Since \(U_\lambda\in\N_{\lambda,\R}\), the fibre maximality part of \cref{Prop2.15} gives
		\[
		\Gamma_{\lambda,\R}(w)
		=
		J_{\lambda,\R}(U_\lambda)
		=
		d_\lambda^\infty .
		\]
		Therefore
		\[
		\limsup_{n\to\infty}d_{\lambda_n}^\infty
		\le
		d_\lambda^\infty .
		\]
		Together with lower semicontinuity, this proves continuity.
		
		Finally, for every \(\lambda\in I\),
		\[
		d_\lambda^\infty
		=
		J_{\lambda,\R}(U_\lambda)
		-
		\frac12J'_{\lambda,\R}(U_\lambda)[U_\lambda]
		=
		\int_\R H(|U_\lambda|)\,\dd x .
		\]
		By \eqref{eq2.3}, the integrand is nonnegative and is strictly positive on the set where \(U_\lambda\ne0\).  Hence
		\[
		d_\lambda^\infty>0
		\qquad(\lambda\in I).
		\]
		Since \(I\) is compact and \(\lambda\mapsto d_\lambda^\infty\) is continuous,
		\[
		\inf_{\lambda\in I}d_\lambda^\infty>0 .
		\]
	\end{proof}
	
	\begin{proof}[Proof of \cref{Prop3.1}]
		For each \(\lambda\in I\), choose a least-energy critical point given by \cref{ThmA.1}.  Apply \cref{LemA.2} and \cref{LemA.4} to apply a translation and multiply by a constant phase so that the first component is even and the second component is odd.  In particular,
		\[
		U_{\lambda,2}(0)=0 .
		\]
		The uniform exponential estimate is \cref{LemA.7}, and the compactness is \cref{LemA.8}.  The continuity of the full-line level and the positivity of \(d_I^*\) follow from \cref{LemA.9}.  This proves all assertions of \cref{Prop3.1}.
	\end{proof}
	
	\section*{Acknowledgment}
	We express our gratitude to the anonymous referee for their careful review of our manuscript and the valuable feedback provided for its enhancement.

	\medskip
	{\bf Funding:} This work is supported by National Natural Science Foundation of China (12301145, 12261107, 12561020, 11961081) and Yunnan Fundamental Research Projects (202401AU070123, 202601AT070048). 
	
	\textbf{Author contributions:}
	Zhipeng Yang: Conceptualization, Methodology, Writing--original draft,
	Writing--review and editing.
	
	\medskip
	{\bf Data availability:}  Data sharing is not applicable to this article as no new data were created or analyzed in this study.
	
	\medskip
	{\bf Conflict of Interests:} The author declares that there is no conflict of interest.
	
	\medskip
	{\bf AI assistance statement:}
	The author used AI-assisted tools to support literature searches,
	the checking and revision of mathematical arguments, and language
	editing during the preparation and revision of this manuscript.
	The figures in this manuscript were generated using ChatGPT.
	The author takes full responsibility for the content of this manuscript.

\end{document}